\documentclass[12pt]{amsart}

\usepackage{amsmath}
\usepackage{amssymb}
\usepackage{amsthm}
\usepackage{enumerate}
\usepackage{latexsym}
\usepackage{algorithmic}
\usepackage[ruled]{algorithm}
\usepackage{graphicx}
\usepackage{tikz}
\usepackage[colorlinks]{hyperref}

\usepackage[latin1]{inputenc}
\usepackage[english]{babel}
\usepackage{color}

\newtheorem{theorem}{Theorem}[section]
\newtheorem{corollary}[theorem]{Corollary}
\newtheorem{proposition}[theorem]{Proposition}
\newtheorem{lemma}[theorem]{Lemma}
\newtheorem{notation}[theorem]{Notation}

\newtheorem*{thmIntroA}{Theorem A}
\newtheorem*{thmIntroB}{Theorem B}

\theoremstyle{definition}
\newtheorem{definition}[theorem]{Definition}
\newtheorem{example}[theorem]{Example}
\newtheorem{remark}[theorem]{Remark}

\newcommand{\minReg}[1]{m_{#1}}
\newcommand{\minRho}[2]{m^{#2}_{#1}}
\newcommand{\reg}{\textnormal{reg}}
\newcommand{\sat}{\textnormal{sat}}
\newcommand{\gin}{\textnormal{gin}}
\newcommand{\ch}{\textnormal{ch}}

\newcommand{\PP}{\mathbb{P}}
\newcommand{\expLift}{expanded lifting}
\newcommand{\innesto}{ideal graft}
\renewcommand{\tilde}{\widetilde}
\renewcommand{\bar}{\overline}
\newcommand{\Proj}{\textnormal{Proj}\,}

\begin{document}

\title[Minimal Castelnuovo-Mumford regularity for a given Hilbert polynomial]{Minimal Castelnuovo-Mumford regularity for a given Hilbert polynomial}\thanks{The first and fourth authors were supported by the PRIN 2010-11  \emph{Geometria delle variet\`a algebriche}, cofinanced by MIUR (Italy)}

\author[F. Cioffi]{F. Cioffi}
\address{Dipartimento di Matematica e Applicazioni dell'Universit\`{a} di Napoli Federico II,\\   via Cintia, 80126  Napoli, Italy
         }
\email{\href{mailto:francesca.cioffi@unina.it}{francesca.cioffi@unina.it}}

\author[P.Lella]{P. Lella}
\address{Dipartimento di Matematica dell'Universit\`{a} di Torino\\ 
         Via Carlo Alberto 10, 
         10123 Torino, Italy}
\email{\href{mailto:paolo.lella@unito.it}{paolo.lella@unito.it}}

\author[M. G. Marinari]{M.G. Marinari}
\address{Dipartimento di Matematica dell'Universit\`{a} di Genova,\\   via Dodecaneso 35, 16146 Genova, Italy
         }
\email{\href{mailto:marinari@dima.unige.it}{marinari@dima.unige.it}}

\author[M. Roggero]{M. Roggero}
\address{Dipartimento di Matematica dell'Universit\`{a} di Torino\\ 
         Via Carlo Alberto 10, 
         10123 Torino, Italy}
\email{\href{mailto:margherita.roggero@unito.it}{margherita.roggero@unito.it}}

\begin{abstract}
Let $K$ be an algebraically closed field of null characteristic and $p(z)$ a Hilbert polynomial. We look for the minimal Castelnuovo-Mumford regularity $\minReg{p(z)}$ of closed subschemes of projective spaces over $K$ with Hilbert polynomial $p(z)$. Experimental evidences led us to consider the idea that $m_{p(z)}$ could be achieved by schemes having a suitable {\em minimal Hilbert function}. We give a constructive proof of this fact. Moreover, we are able to compute the minimal Castelnuovo-Mumford regularity $\minRho{p(z)}{\varrho}$ of schemes with Hilbert polynomial $p(z)$ and given regularity $\varrho$ of the Hilbert function, and also the minimal Castelnuovo-Mumford regularity $m_u$ of schemes with Hilbert function $u$.

These results find applications in the study of Hilbert schemes. They are obtained by means of \emph{minimal Hilbert functions} and of two new constructive methods which are based on the notion of growth-height-lexicographic Borel set and called \emph{ideal graft} and \emph{extended lifting}. 
\end{abstract}

\keywords{Castelnuovo-Mumford regularity, Hilbert polynomial, regularity of a Hilbert function, minimal function, Borel ideal, closed projective subschemes}
\subjclass[2000]{14Q99, 68W30, 11Y55}
\maketitle


\section*{Introduction}

This paper deals with the Castelnuovo-Mumford regularity $\reg(X)$ of closed subschemes $X$ of projective spaces over an algebraically closed field $K$ of characteristic zero, with a given Hilbert polynomial $p(z)$. Hence, a scheme will be understood to be a closed subscheme of a projective space. We do not establish the dimension of the projective space in which the schemes are embedded, although we will be able to give information also for every given dimension.

As we can read in literature, $\reg(X)$ is \lq\lq one of the fundamental invariants in Commutative Algebra and Algebraic Geometry\rq\rq (see \cite{Br}) and can be also considered as \lq\lq a measure of the complexity of computing Gr\"obner bases\rq\rq \cite{BayerMumford}. Both these aspects of the Castelnuovo-Mumford regularity are present in this paper. Thus,  besides the classical definitions of the Castelnuovo-Mumford regularity in terms of ideal sheaf cohomology and of syzygies, we will also recall its characterization in terms of the generic initial ideal with respect to the degree reverse lexicographic term order (see Definition \ref{def:regularity} and Proposition \ref{borel}). 

In his famous paper \cite{Go}, G. Gotzmann finds a sharp upper bound for the Castelnuovo-Mumford regularity of schemes with a given Hilbert polynomial. This upper bound, called {\em Gotzmann number}, is fundamental in the study of Hilbert schemes and determines a finite range of possible values for the Castelnuovo-Mumford regularity, given a Hilbert polynomial. Here, we tackle the following question posed by E.~Ballico in a private conversation: {\em what is the minimal value $m_{p(z)}$ of the Castelnuovo-Mumford regularity of closed subschemes of projective spaces with given Hilbert polynomial $p(z)$?}

Here, we provide a complete answer to the posed question, because we give a sharp lower bound for the Castelnuovo-Mumford regularity of schemes with a given Hilbert polynomial, thus narrowing the range in which the Castelnuovo-Mumford regularity of a scheme with a given Hilbert polynomial can vary. This additional information can have many significant applications. To offer an example, it can be applied to the study described in \cite{BBR}, from which the question considered in our manuscript arised. In that paper explicit equations and inequalities defining the open locus of the Hilbert scheme of points with bounded regularity are obtained: our result tells when such an open subset is non-empty.

To reach our aim, we study the {\em minimal functions} and develop two new constructive methods, which we think can have further applications also in more general contexts as, for example, the study of liftings and that of general hypersurface sections.

In some cases, to find the value of $m_{p(z)}$ is almost immediate. For example, for a constant Hilbert polynomial $p(z)=d>1$, we obtain $m_{p(z)}=2$. Indeed, having $K$ an infinite cardinality, there exists a scheme $X\subset \PP_K^{d-1}$ of $d$ points in generic position, that has Hilbert function $(1,d,d,\ldots)$ (for instance, see \cite[Theorem 2.5]{GMR}).

However, in general, the posed question has not such a straightforward answer.
Some preliminary results highlight the role played in this context by the regularity $\varrho_X$ of the Hilbert function $H_X$ of the scheme $X$. Indeed, they suggest that, if $\reg(X) = m_{p(z)}$, then the integers $\varrho_X$ and $H_X(\varrho_X-1)$ are small with respect to the corresponding ones of other schemes with greater Castelnuovo-Mumford regularities (Lemma \ref{CMR} and Proposition \ref{newreg}). These facts together with experimental evidences encourage us to look for an answer by focusing on the following two key ideas: to find a suitable notion of minimal function and to consider the minimal Castelnuovo-Mumford regularity $m^{\varrho}_{p(z)}$ of schemes, by choosing also the regularity $\varrho$ of the Hilbert function. We gained the experimental evidences supporting our ideas by a pencil-and-paper work and by  using the applet that is available at \href{http://www.personalweb.unito.it/paolo.lella/HSC/Borel-fixed_ideals.html}{\texttt{http://www.personalweb.unito.it/paolo.lella/HSC/Borel-fixed\_ideals.html}} \href{http://www.personalweb.unito.it/paolo.lella/HSC/Borel-fixed_ideals.html} \ to compute suitable Borel-fixed ideals. Some of our experiments are reproduced here as examples. 

We develop the study of minimal functions with given regularity generalizing a construction that L.~G.~Roberts introduced in a special case (see \cite{R}). The minimal functions have very nice properties due to their purely combinatorial structure (see Theorem \ref{th:first derivative}) and turn out to be sufficient for a complete description of minimal Castelnuovo-Mumford regularities, in the following sense. 

\begin{thmIntroA}
Let $S(p(z),\varrho)$ be the set of schemes with given Hilbert polynomial $p(z)$ and regularity $\varrho$ of their Hilbert functions. If $S(p(z),\varrho)$ is non-empty, then it contains a scheme $X$ such that both its Hilbert function $H_X$ and its Castelnuovo-Mumford regularity $\reg(X)$ are minimal for schemes in $S(p(z),\varrho)$.
\end{thmIntroA}

In \cite[Theorem 3.2]{CS}, the authors observe that Theorem A holds on fields of any characteristic.

In order to find $m_{p(z)}$, we compare the integers $m^{\varrho}_{p(z)}$, as $\varrho$ varies among all values for the regularities of Hilbert functions of schemes with the given Hilbert polynomial $p(z)$. Our analysis allows us to prove the following statement, which is even stronger than the one suggested by the preliminary results.

\begin{thmIntroB}
Let $\varrho$ and $\varrho'$ be the regularities of two Hilbert functions of schemes with Hilbert polynomial $p(z)$. Then,
\begin{itemize}
\item $\varrho < \varrho' \quad \Longrightarrow \quad \minRho{p(z)}{\varrho} \leqslant \minRho{p(z)}{\varrho'}$;
\item if $\varrho <m_{p(z)}-1$ then $\minRho{p(z)}{\varrho}=m_{p(z)}$; especially, $m_{p(z)}=\minRho{p(z)}{\varrho}$ for the minimal possible $\varrho$;
\item if $\varrho\geq m_{p(z)}-1$ then $\minRho{p(z)}{\varrho}=\varrho+1$ or $\minRho{p(z)}{\varrho}=\varrho+2$.
\end{itemize}
\end{thmIntroB}

The above results are collected in Theorem \ref{primo teorema} and Corollary \ref{cor:finale innesto}. Their  proofs are based on two new constructive methods we conceived by exploiting the features of the growth-height-lexicographic Borel sets, which have been introduced by D.~Mall in \cite{Mall1,Mall2}. 
Both these methods combine together properties of two schemes $X_1$ and $X_2$, returning a new scheme with Hilbert function and Castelnuovo-Mumford regularity depending on those of $X_1$ and $X_2$. They differ by the numerical hypotheses they require to be applied. 

The first method, called \emph{ideal graft} constructs a scheme with the Hilbert function of $X_1$ and the Castelnuovo-Mumford regularity depending on that of $X_2$. This method has a special role in the comparison of the integers $m^{\varrho}_{p(z)}$ (Theorem \ref{innesto} and Corollary \ref{cor:uso innesto}). 

The\hfill second\hfill method,\hfill called\hfill \emph{\expLift},\hfill computes\hfill a\hfill scheme\hfill having\\ Hilbert function of $X_1$ and hyperplane section with the Hilbert function of $X_2$. Moreover, the expanded lifting has a total control on the Castelnuovo-Mumford regularity of the constructed scheme (see Theorem \ref{construction}).

Our final result gives a recursive procedure to compute the minimal Castelnuovo-Mumford regularity $m_u$ of a scheme with a Hilbert function $u$ and provides very strict lower and upper bounds for $m_u$ (see Theorem \ref{construction-main}). This computation allows to determine $m^{\varrho}_{p(z)}$, $m_{p(z)}$ and also the minimal  Castelnuovo-Mumford regularity of schemes with Hilbert polynomial $p(z)$ that are embedded in a given projective space $\mathbb P^n_K$.

The paper is organized in the following way. In section 1, we set some notation about monomial ideals. In section 2, we recall some basic definitions and classical results about Hilbert functions and Castelnuovo-Mumford regularity. We also introduce some new preliminary results (Propositions \ref{newreg} and \ref{prop:varrho}).
In section 3, we introduce the minimal functions with given regularities, studying their combinatorial properties in relation with the existence of schemes with given Hilbert polynomials and Hilbert function regularity (see Theorem \ref{th:first derivative}). Moreover, we describe a closed formula for the minimum $\varrho$ such that $S(p(z),\varrho)\not=\emptyset$ (see Proposition \ref{prop:calcolo bar varrho}).

In section 4, we recall the notion of growth-height-lexicographic Borel set, which is crucial for the description of the two constructive methods ideal graft and expanded lifting, which we propose and describe in sections 5 and 6, respectively. In section 7, we provide the proofs of our main results.

As a natural consequence of the computational point of view that supported this paper, all our results produce some constructive methods. In the appendix A, we describe the main algorithmic procedures to compute the minimal Castelnuovo-Mumford regularities that arise from our exposition. 


\section{General setting}

Let $K$ be an algebraically closed field of characteristic $0$, $S:=K[x_0\ldots,x_n]=\bigoplus_{t\geq 0} S_t$ be the ring of polynomials over $K$ in $n+1$ variables with $x_0<x_1<\ldots <x_n$, where $S_t$ is the $K$-vector space of the homogeneous polynomials of degree $t$, and $\mathbb P^n_K = \Proj S$ be the $n$-dimensional projective space over $K$. 

For a subset $M\subseteq S$ we set $M_t:=M\cap S_t$. For a homogeneous ideal $I$ of $S$, we denote by  $I_{\leq t}$ the ideal generated by the polynomials 
of $I$ of degree $\leq t$.

A {\em term} of $S$ is a power product $x^{\alpha}:=x_0^{\alpha_0} x_1^{\alpha_1} \ldots x_n^{\alpha_n}$, where ${\alpha_0}, {\alpha_1},\ldots,\alpha_n$ 
are non negative integers, and  $\mathbb T:=\{x_0^{\alpha_0} x_1^{\alpha_1}\ldots x_n^{\alpha_n} \ \vert \ (\alpha_0, \alpha_1\ldots,\alpha_n) \in 
\mathbb N^{n+1}\}$ is the multiplicative monoid of all terms of $S$. For a monomial ideal $J$, we denote by $\mathcal N(J)$ the {\em sous-escalier} 
of $J$, i.e.~the set of all terms outside $J$. 

In our setting, the graded term orders deglex and degrevlex are defined saying that, given two terms $x^{\alpha}$ and 
$x^{\beta}$ of $\mathbb T$ of the same degree $t$, $x^{\alpha}$ is less than $x^{\beta}$ with respect to: 
\begin{itemize}
\item[] {\em deglex} order, if $\alpha_k<\beta_k$, where $k=\max\{i\in\{0,\ldots,n\} : \alpha_i\not= \beta_i \}$; 
\item[] {\em degrevlex} order, if $\alpha_h>\beta_h$, where $h=\min\{i\in\{0,\ldots,n\} : \alpha_i\not= \beta_i \}$.
\end{itemize}

Given a degree $t$, a set $B\subset \mathbb T_t$ is a {\em lex-segment} if it consists of the $\vert B\vert$ highest terms of $\mathbb T_t$ with respect to the deglex order. Given a subset $A\subset\mathbb T_t$, a {\em lex-segment in $A$} is the intersection of a lex-segment of $\mathbb T_t$ and $A$.  A monomial ideal $J$ is a {\em lex-segment ideal} if $J_t$ is a lex-segment for every integer $t$.


\section{Hilbert function and Castelnuovo-Mumford regularity}

For a homogeneous ideal $I\subset S$, the Hilbert function of the graded algebra $S/I$ is denoted by $H_{S/I}$ and its Hilbert polynomial is denoted by $p_{S/I}(z)\in \mathbb{Q}[z]$. The \emph{regularity of the Hilbert function of $S/I$} is the integer $\varrho_{S/I}:=\min\{t\in \mathbb N \ \vert \ H_{S/I}(t')=p_{S/I}(t'), \forall \ t'\geq t\}$.

We\hfill set\hfill $\Delta^0 H_{S/I}(t):=H_{S/I}(t)$\hfill and,\hfill when\hfill $S/I$\hfill is\hfill not\hfill Artinian,\hfill we\hfill let\\ $\Delta^i H_{S/I}(0):=1$ and $\Delta^i H_{S/I}(t):= \Delta^{i-1}H_{S/I}(t)-\Delta^{i-1}H_{S/I}(t-1)$,  for each $1\leq i\leq \dim_{\text{Krull}}(S/I)$ and for $t>0$, calling it {\em i-th derivative} of $H$; we use an analogous notation for Hilbert polynomials. 

We define the function $\Sigma H_{S/I}: \mathbb{N} \longrightarrow\mathbb{N}$ letting $\Sigma H_{S/I}(0):=1$ and $\Sigma H_{S/I}(t):= \Sigma H_{S/I}(t-1)+H_{S/I}(t)$ for each $t\geq 1$, and call it {\em integral} of $H_{S/I}(t)$; we use an analogous notation for Hilbert polynomials. 

Recall that, given two positive integers $a$ and $t$, the {\em binomial expansion of $a$ in base $t$} is the unique writing
\begin{equation}
\label{espansione}
\begin{array}{lllll}
a & = & \binom{k(t)}{t} + \binom{k(t-1)}{t-1} + \ldots + \binom{k(j)}{j} &=: &a_t
\end{array}
\end{equation} 
where $k(t)> k(t-1)>\ldots > k(j)\geq j\geq 1$. We use the convention that a binomial coefficient $\binom{n}{m}$ is null whenever either $n<m$ or $m<0$ 
and $\binom{n}{0}=1$, for all $n\geq 0$. Referring to \cite{Ro}, we let 
$$\begin{array}{ll}
(a_t)^+_+&:=\binom{k(t)+1}{t+1} + \binom{k(t-1)+1}{t} + \ldots + \binom{k(j)+1}{j+1},  \text{ and }\\
(a_t)^-_-&:=\binom{k(t)-1}{t-1} + \binom{k(t-1)-1}{t-2} + \ldots + \binom{k(j)-1}{j-1}.
\end{array}$$
For convenience, we also set $((a_t)^-_-)^+_+:= (((a_t)^-_-)_{t-1})^+_+$. By an easy computation one gets (for example, see \cite[Proposition 4.9]{Ro})
\begin{equation}\label{doppia azione}
((a_t)^-_-)^+_+=\left\{\begin{array}{ll}
a, &\text{if } j>1 \\
a + k(2)-k(1), &\text{if } j=1
\end{array}\right.,\end{equation}
hence, $((a_t)^-_-)^+_+ \geq a$. Moreover, \cite[Prop. 4.3(a) and 4.6(b)]{Ro} 
\begin{equation}\label{disuguaglianze robbiano}
((a+1)_t)^+_+ = (a_t)^+_+ +1 + k(1)  \hbox{ and } ((a+1)_t)^-_- = \left\{ \begin{array}{cc} 
(a_t)^-_- +1 &\hbox{if } j>1\\
(a_t)^-_-   &\hbox{if } j=1. \end{array}\right.
\end{equation}

A numerical function $H:\mathbb N \rightarrow \mathbb N$ is {\it admissible} (or an {\em O-sequence}) if $H(0)=1$ and $H(t+1)\leq (H(t)_t)^+_+$. In particular, 
if $H(t)=0$, then $H(t+h)=0$ for every $h>0$. We will say that $H$ is {\em admissible in $\bar t$} if $H(\bar t+1)\leq (H(\bar t)_{\bar t})^+_+$.

Due to \cite{Ma}, a numerical function $H$ is admissible if, and only if, it is the Hilbert function of $S/I$, for a suitable homogeneous ideal $I$. 

\begin{definition}\label{def:regularity}
A homogeneous ideal $I$ is $m$-regular if the $i$-th syzygy module of $I$ is generated in degree $\leq m+i$. The {\em regularity} $\reg(I)$ of $I$ is the smallest integer $m$ for which $I$ is $m$-regular. 
The saturation of $I$ is $I^{\sat}:=\{f \in S \ \vert \ \forall \ i\in {0,\ldots,n}, \exists\ k_i : x_i^{k_i}f \in I\}$ and $I$ is saturated if $I=I^{\sat}$. 

With the common notation of the ideal sheaf cohomology, given a scheme $X\subset \mathbb{P}^n_K$ and its (saturated) defining ideal $I=I(X)$, the {\em Castelnuovo-Mumford regularity} of $X$ is $\reg(X):=\min\{t\in \mathbb N\ \vert \ H^i(\mathcal I_X(t'-i))=0, \forall \ t'\geq t, \forall i > 0 \}$ and it is equal to $\reg(I)$. Also, we set $H_X(t):=H_{S/I}(t)$, $p_X(z):=p_{{S}/{I}}(z)$, $\varrho_X:=\varrho_{S/I}$. 
\end{definition}

If the dimension of a scheme $X\subset \mathbb P^n_K$ is $k>0$, let $h\in S_1$ be a general linear form that is not a zero-divisor on $S/I$ and $J=(I,h)$, where $I:=I(X)$. It is well known that the first derivative of the Hilbert function of a projective scheme is admissible, and the converse is true by \cite[Corollary 3.4]{GMR}. Indeed, let $Z\subset{\mathbb P}^{n-1}_K$ be the scheme of dimension $k-1$ defined by the saturated ideal $J^{\sat}/(h) = (I,h)^{\sat}/(h)$, i.e.~the general hyperplane section of $X$. Since the linear form $h$ is not a zero-divisor on $S/I$, we have the short exact sequence
$0 \rightarrow (S/I)_{t-1} \xrightarrow{\cdot h} (S/I)_t \rightarrow (S/J)_t \rightarrow 0$ \ 
that gives $H_{S/J}(t) = \Delta H_{S/I}(t)$, in particular $p_{{S}/{(I,h)}}=\Delta p_{{S}/{I}}$, and then $\varrho_{S/J}=\varrho_X+1$, so $\Delta H_X(t)=H_{S/J}(t)\geq H_Z(t)$ for every $t$ and $H_{S/J}(t)=H_Z(t)$ for $t\geq \max\{\varrho_Z,\varrho_{S/J}\}$. 

This relation between the first derivative of the Hilbert function of $X$ and the Hilbert function of $Z$ suggests to consider the following partial order.

\begin{definition}[{\cite{R}}]\label{order} 
Given two sequences of integers $A=(a_i)_{i\in \mathbb N}$ and $B=(b_i)_{i\in \mathbb N}$, we let $A\preceq B$, if $a_i\leq b_i$ for every index $i$. 
\end{definition}

\begin{remark}\label{derivata-integrale}
Let $f$ be a Hilbert function with Hilbert polynomial $p(z)$. If $g$ is an other Hilbert function with Hilbert polynomial $\Delta p(z)$ such that $g\preceq \Delta f$, then $\sum g \preceq f$. Moreover, $\sum g$ has Hilbert polynomial $p(z)-c$, where $c$ is a non-negative integer. If $\bar t$ is the minimal integer such that $g(\bar t) < \Delta f(\bar t)$, then $\sum g(t) < f(t)$ for every $t\geq \bar t$.
\end{remark}

By cohomological arguments, one gets $\reg(X)\geq \reg(Z)$ and $\reg(X)\geq \varrho_X+1$, for every scheme $X$. In particular, the following result tells how $\varrho_X$ and $\reg(Z)$ determine $\reg(X)$.

\begin{lemma}[{\cite[Lemma 3.6]{CMR}}]\label{CMR}
With the above notation, we have $\reg(X)=\max\left\{\reg(Z),\varrho_X+1\right\}$.
\end{lemma}

For a Cohen-Macaulay scheme $W\subset \mathbb P_K^n$ of dimension $k$ and degree $d$, the Castelnuovo-Mumford regularity is $\reg(W)=\varrho_W+k+1 \leq d$ (e.g.~\cite{BS}). Moreover, if $X\subset \mathbb P_K^n$ is a scheme of odd dimension $k$ with the same Hilbert function as a Cohen-Macaulay scheme, then $\reg(X)>\varrho_X+1$ \cite[Proposition 2.4]{CDG}, even if the characteristic of $K$ is positive. By exploiting the proof of that result we obtain the following more general statement.

\begin{proposition}\label{newreg}
If $X\subset \mathbb P_K^n$ is a scheme with $H_X(\varrho_X-1)>p_X(\varrho_X-1)$, then $\reg(X)>\varrho_X+1$.
\end{proposition}

\begin{proof}
Let $Z$ be a general hyperplane section of $X$. By the hypothesis we get 
$H_Z(\varrho_X)\leq \Delta H_X(\varrho_X)< \Delta p_X(\varrho_X)=p_Z(\varrho_X)$. 
Hence, we obtain $\varrho_Z>\varrho_X$ and $\reg(X)\geq \reg(Z)\geq \varrho_Z+1 > \varrho_X+1$.
\end{proof}

Polynomials $p(z)\in \mathbb Q[z]$ that are Hilbert polynomials of schemes are called {\em admissible} and are completely characterized in \cite{H66}.

The {\em Gotzmann number} $r$ of an admissible polynomial $p(z)$ is the best upper bound for the Castelnuovo-Mumford regularity of a scheme having $p(z)$ as Hilbert polynomial and is computable by using the following unique form of an admissible polynomial:
\begin{equation}\label{eq:b}
\begin{array}{lll}
p(z) & = & \binom{z+k_1}{k_1}+\binom{z+k_2-1}{k_2}+\ldots+\binom{z+k_r-(r-1)}{k_r},
\end{array}
\end{equation}
with $r,k_i\in\mathbb{N}, k_1\geq k_2\geq \ldots \geq k_r\geq 0$ \cite{Go}. We refer to \cite{Gr} for an overview of these arguments. For a constant polynomial $p(z)=d$ we have $r=d$.

A polynomial with Gotzmann number $r=1$ is of type $p(z)=\binom{h+z}{z}$ and is the Hilbert polynomial of a linear variety $X$. From now, $p(z)$ is an admissible Hilbert polynomial with Gotzmann number $r>1$.

\begin{notation}\label{varie}
We will need the following notation:
\begin{itemize}
\item $S(p(z),\varrho)$ is the set of schemes with Hilbert polynomial $p(z)$ and regularity $\varrho$ of the Hilbert function.
\item $F(p(z),\varrho)$ is the set of the Hilbert functions of the schemes in $S(p(z),\varrho)$ and we let $F(p(z)):=\cup_{\varrho} F(p(z),\varrho)$. 
\item For every $u\in F(p(z),\varrho)$, $m_u$ is the minimal possible Castelnuovo-Mumford regularity of a scheme with Hilbert function $u$;
\item $M(p(z),\varrho):=\{m_u : u\in F(p(z),\varrho)\}$ and $m^{\varrho}_{p(z)}:=\min M(p(z),\varrho)$.
\item $\bar\varrho_{p(z)}$ is the minimal integer $\varrho$ such that $F(p(z),\varrho)\not= \emptyset$.
\item $\varrho_{p(z)}$ is defined in \eqref{varrho p(z)} and its meaning is explained by Propositions \ref{prop:varrho} and \ref{prop:minimalfunction}.
\item $\bar f$ and $\bar m$ are defined in \eqref{barra m}.
\end{itemize}
\end{notation}


\begin{proposition}\label{prop:varrho}
The set $\Pi_{p(z)}:=\{1 \leq t \leq r : (p(t+h)_{t+h})^+_+ \geq p(t+h+1)\geq 1, \forall \ h\geq 0\}$ is non-empty and the minimal regularity of the admissible functions with Hilbert polynomial $p(z)$ is lower bounded by 
\begin{equation}\label{varrho p(z)}
\varrho_{p(z)}:=\left\{\begin{array}{ll}
0, & \text{ if }\min\ \Pi_{p(z)}=1 \text{ and } p(0)=1 \\
\min\ \Pi_{p(z)}, & \text{ otherwise.}
\end{array}\right.
\end{equation} 
Moreover,
\begin{enumerate}[(i)]
\item\label{it:varrho_i} $\bar \varrho_{p(z)} \geq \max\{\varrho_{p(z)}, \varrho_{\Delta p(z)}-1\}$ 
\item\label{it:varrho_ii} if $\varrho_{p(z)}> 0$, then $\varrho_{p(z)+c} \leq \varrho_{p(z)}$ and $\bar \varrho_{p(z)+c} \leq \bar\varrho_{p(z)}$ for every $c>0$.
\end{enumerate}
\end{proposition}

\begin{proof}
Observe that $(p(r+h)_{r+h})^+_+ = p(r+h+1)$, for every $h\geq 0$, by the well-known fact that the Gotzmann number $r$ of an admissible polynomial $p(z)$ is the regularity of the unique saturated lex-segment ideal with Hilbert polynomial $p(z)$ (for example, see \cite[Proposition 3.7]{Gr}). Hence, the set $\Pi_{p(z)}$ is non-empty.

Let $f$ be an admissible function with Hilbert polynomial $p(z)$ and regularity $\varrho$. Then, for every $s\geq \varrho$, we have $(p(s)_{s})^+_+ =(f(s)_{s})^+_+ \geq f(s+1)_{s+1} = p(s+1)_{s+1}\geq 1$ and $s\in \Pi_{p(z)}$. Hence, the regularity of a Hilbert function with Hilbert polynomial $p(z)$ is lower bounded by $\varrho_{p(z)}$.

For the assertion \emph{(\ref{it:varrho_i})}, if $g$ is a function of $F(p(z),\bar\varrho_{p(z)})$, then $\bar\varrho_{p(z)}\geq \varrho_{p(z)}$ and $\Delta g$ is admissible with regularity $\bar\varrho_{p(z)}+1$, hence $\varrho_{\Delta p(z)}\leq \bar \varrho_{p(z)}+1$. For \emph{(\ref{it:varrho_ii})}, we obtain both $\varrho_{p(z)+c} \leq \varrho_{p(z)}$ and $\bar \varrho_{p(z)+c} \leq \bar\varrho_{p(z)}$ by $\Delta (p(z)+c)=\Delta p(z)$ and by the first formula of \eqref{disuguaglianze robbiano}, because $((p(t)+c)_{t})^+_+ \geq (p(t)_t)^+_+ +c\geq p(t+1)+c$, for every $t\geq \varrho_{p(z)}$.
\end{proof}


\section{Minimal functions}\label{sec:minimalFunctions}

In this section, we will prove there exists the minimum of the Hilbert functions with Hilbert polynomial $p(z)$ and regularity $\varrho$, for every $\varrho\geq \varrho_{p(z)}$, with respect to the partial order $\preceq$ of Definition \ref{order}. Each of these minimal functions will be of the following type.

\begin{definition} \label{def:minimal functions}
For every $\varrho\geq \varrho_{p(z)}$, we let
\begin{equation} \label{alla roberts}
f_{p(z)}^{\varrho}(t):=\left\{\begin{array}{ll}
p(t), & \text{ if } t\geq \varrho \\
(f_{p(z)}^{\varrho}(t+1)_{t+1})^-_- , & \text{ otherwise}
\end{array}\right. 
\end{equation} 
\begin{equation}\label{def: new function} g_{p(z)}^{\varrho}(t):=\left\{\begin{array}{ll} p(t), & \text{ if } t\geq \varrho \\ p(t)+1, &\text{ if } t=\varrho-1 \\(g_{p(z)}^{\varrho}(t+1)_{t+1})^-_- , & \text{ otherwise.}\end{array}\right.\end{equation} 
\end{definition}

\begin{proposition}\label{prop:minimalfunction}
For every integer $\varrho_{p(z)}\leq \varrho\leq r-1$, $f_{p(z)}^{\varrho}$ is admissible, hence $\varrho_{p(z)}$ is the minimal regularity of Hilbert functions with Hilbert polynomial $p(z)$. Moreover, $f_{p(z)}^\varrho\succeq f_{p(z)}^{\varrho+1}$ and $f_{p(z)}^{\varrho}$ is the minimal Hilbert function with Hilbert polynomial $p(z)$ and regularity $\leq \varrho$. In particular, $f_{p(z)}^{\varrho_{p(z)}}$ has regularity $\varrho_{p(z)}$ and, if $\varrho>\varrho_{p(z)}$, then $f_{p(z)}^{\varrho}(\varrho-1)\leq p(\varrho-1)$.
\end{proposition}

\begin{proof} By Proposition \ref{prop:varrho}, the regularity of every Hilbert function with Hilbert polynomial $p(z)$ is lower bounded by $\varrho_{p(z)}$.

With the notation introduced in formula (\ref{espansione}), for a positive integer $a$, $(a_t)^-_-$ is the smallest integer $b$ such that $a\leq (b_{t-1})^+_+$, thanks to formula \eqref{doppia azione}. Hence, by construction the numerical function $f_{p(z)}^{\varrho}$ is the minimal Hilbert function with Hilbert polynomial $p(z)$. Moreover, $f_{p(z)}^{\varrho_{p(z)}}$ has regularity $\varrho_{p(z)}$, by definition of $\varrho_{p(z)}$. 
The last assertion holds by construction, also if $\varrho_{p(z)}=0$ and $\varrho=1$, because in this case $f_{p(z)}^{\varrho_{p(z)}}(0)=p(0)=1$.
\end{proof}

\begin{remark}\label{rm:minimal Roberts}
In \cite{R}, L.G.~Roberts introduces the minimal Hilbert function we denote by $f^{r-1}_{p(z)}$, showing that it is the Hilbert function of the so-called tight fan constructed by Hartshorne and also the Hilbert function of the saturated lex-segment ideal with the given Hilbert polynomial \cite[Lemma 5.1 and Theorem 5.3]{R}. This function attains the equality in the Hyperplane Restriction Theorem due to M. Green \cite[Theorem 3.4]{Gr}. 
\end{remark}

\begin{remark}\label{varrho'}
(1) For every $\varrho\geq r$, we can set $f_{p(z)}^{\varrho}:=f_{p(z)}^{r-1}$, because $(p(r+h)_{r+h})^+_+ = p(r+h+1)$, for every 
$h\geq 0$, \cite[Proposition 3.7]{Gr}.

(2) If $p(z)=d$ is a constant polynomial, for every $1\leq \varrho\leq d-1$, the corresponding minimal function $f_d^{\varrho}$ has regularity equal to $\varrho$ and every scheme with this Hilbert function has Castelnuovo-Mumford regularity $\varrho+1$, because it is Cohen-Macaulay. 

(3) If the regularity of $f_{p(z)}^{\varrho}$ is $\varrho'<\varrho$, then $f_{p(z)}^{\varrho'}=f_{p(z)}^{\varrho}$. For example, if we take the admissible polynomial $p(z)=5z-3$ with Gotzmann number $r=7$, we obtain $\varrho_{p(z)}=3$ and $f_{p(z)}^3=f_{p(z)}^{4}$ because $(p(4)_4)^-_-=p(3)$, so the regularity of $f_{p(z)}^{4}$ is $3<\varrho=4$. On the other hand, we have $f_{p(z)}^{4}\not= f_{p(z)}^{5}$ because $p(4)=17$ and $f_{p(z)}^{5}(4)=16$. Anyway, there exists the minimal function with regularity $\varrho$ even when the regularity of $f^{\varrho}_{p(z)}$ is strictly lower than $\varrho$. 
\end{remark}

\begin{proposition}\label{minimal paolo}
For every $\varrho\geq \varrho_{p(z)}$, the function $g_{p(z)}^{\varrho}$ is admissible and, if the regularity of $f^\varrho_{p(z)}$ is $\varrho'<\varrho$, then $g_{p(z)}^{\varrho}$ is the minimal admissible function with regularity $\varrho$ and Hilbert polynomial $p(z)$. Moreover, $f^\varrho_{p(z)}\preceq g_{p(z)}^{\varrho}$.
\end{proposition}

\begin{proof}
It is enough to argue as in the proof of Proposition \ref{prop:minimalfunction}.
\end{proof}

Now, as we announced at the beginning of this section, for every $\varrho\geq \varrho_{p(z)}$ we have the minimal Hilbert function with regularity $\varrho$, that is either $f_{p(z)}^{\varrho}$ or $g_{p(z)}^{\varrho}$. Anyway, we are interested in Hilbert functions of schemes, that are admissible functions with also admissible first derivative.

\begin{example}\label{ex:first derivative}
If we consider $p(z)=6z^2-18z+37$ we get $\varrho_{p(z)}=1$, $\Delta p(z)=12z-24$ and $\varrho_{\Delta p(z)}=5$, so $f_{p(z)}^1$ is admissible but its first derivative $(1,24,12z-24)$ is not, because of its behavior in the degrees $\varrho_{p(z)}\leq t \leq \varrho_{\Delta p(z)}$. For example, $\Delta f_{p(z)}^1$ is not admissible in $2$ because $\Delta f_{p(z)}^1(2)=0$.
\end{example}

We will show that the construction of a minimal function $f^{\varrho}_{p(z)}$ guarantees the admissibility of the first derivative of the function in the degrees strictly lower than $\varrho$. As a consequence, we will prove there exists the minimum of $F(p(z),\varrho)$ with respect to the partial order $\preceq$ of Definition \ref{order}, as soon as $F(p(z),\varrho)$ is non-empty, and this minimal function is either $f_{p(z)}^{\varrho}$ or $g_{p(z)}^{\varrho}$.

The following technical result is used in Lemma \ref{lemma:first derivative} to detect when the first derivative of a minimal function $f^{\varrho}_{p(z)}$ is admissible. 

\begin{lemma}\label{menomeno}
Assume that an integer $a$ with binomial expansion $\binom{k(t)}{t}+ \ldots+ \binom{k(j)}{j}$ in base $t$ can be written also as 
$a=\binom{h(t)}{t}+ \binom{h(t-1)}{t-1}+ \ldots+ \binom{h(j')}{j'}$ with $h(u)>h(u-1)$, for every $j'<u\leq t$, and $k(t)> h(t)$.
Then $a=\binom{k(t)}{t}$ and the two writings of $a$ are 
$\binom{k(t)}{t}$ and $\sum_{i=0}^{t} \binom{k(t)-1-i}{t-i}$, respectively.
\end{lemma}

\begin{proof}
Let $b=k(t)$.\hfill By\hfill the\hfill definition\hfill of\hfill binomial\hfill expansion\hfill we\hfill have\\ $\binom{b}{t} \leq \binom{h(t)}{t}+ \binom{h(t-1)}{t-1}+ \ldots+ \binom{h(j')}{j'}$.
 Replacing $\binom{b}{t}$ by $\sum_{i=0}^{t} \binom{b-1-i}{t-i}$ we get 
$\sum_{i=0}^{t} \binom{b-1-i}{t-i} \leq \sum_{i=0}^{t-j'} \binom{h(t-i) }{t-i}=a$. 
On the other hand, for every $i=0,\dots, t-j'$ we have $\binom{b-1-i}{t-i}\geq \binom{h(t)-i}{t-i}\geq  \binom{h(t-i)}{t-i}$, hence $j'=0$ and by the equality 
between the two members we obtain $a=\binom{b}{t}$.
\end{proof}

\begin{lemma}\label{lemma:first derivative}
\begin{enumerate}[(i)]
\item \label{it:first derivative_i} For every $\varrho\geq \varrho_{p(z)}$ and $1\leq t<\varrho$, $\Delta f^{\varrho}_{p(z)}$ is admissible in $t$, because $\Delta f^{\varrho}_{p(z)}(t)= (\Delta f^{\varrho}_{p(z)}(t+1)_{t+1})^-_-$.
\item \label{it:first derivative_ii} For every $\varrho > \varrho_{p(z)}$ and $t < \varrho$, $\Delta f^{\varrho}_{p(z)}(t) \leq \Delta f^{\varrho-1}_{p(z)}(t)$.
\item \label{it:first derivative_iii} For every $\varrho \geq\max\{\varrho_{p(z)}, \varrho_{\Delta  p(z)}-1\}$,  $\Delta f^{\varrho}_{p(z)}$ is admissible.
\end{enumerate}
\end{lemma}

\begin{proof}
{\em (\ref{it:first derivative_i})} For every $1\leq t\leq \varrho-1$ by construction we have
$f^{\varrho}_{p(z)}(t+1)= \binom{k(t+1)}{t+1}+ \ldots+ \binom{k(j)}{j}$, 
$f^{\varrho}_{p(z)}(t) = \binom{k(t+1)-1}{t}+ \ldots+ \binom{k(j)-1}{j-1}$ and
$f^{\varrho}_{p(z)}(t-1) =\binom{k(t+1)-2}{t-1}+ \ldots+ \binom{k(j)-2}{j-2}$ and obtain
$\Delta f^{\varrho}_{p(z)}(t) = (\Delta f^{\varrho}_{p(z)}(t+1)_{t+1})^-_-$.

Nevertheless, we have to consider also the case the binomial expansion of $f(t)$ in base $t$ is different from the writing obtained by the binomial expansion of $f(t+1)$ in base $t+1$. This is exactly the case considered in Lemma \ref{menomeno}, for which $f(t)=\binom{k(t+1)-1}{t}+ \binom{k(t+1)-2}{t-1}\ldots+ \binom{k(t+1)-(t+1)}{0}=\binom{k(t+1)}{t}$ and the operation that consists in subtracting a unit to both the integers of the binomial coefficients gives the same result.

{\em (\ref{it:first derivative_ii})} By Proposition \ref{prop:minimalfunction}, we have $f^{\varrho}_{p(z)}(\varrho-1) \leq   p(\varrho-1)=f^{\varrho-1}_{p(z)}(\varrho-1)$.

Let $c$ be the non-negative integer such that $f^{\varrho}_{p(z)}(\varrho-1)=p(\varrho-1)-c$. Then, formula \eqref{disuguaglianze robbiano} and the definition of a minimal function give:
$$\Delta f^{\varrho}_{p(z)}(\varrho-1)=f^{\varrho}_{p(z)}(\varrho-1)-f^{\varrho}_{p(z)}(\varrho-2)=p(\varrho-1)-c-((p(\varrho-1)-c)_{\varrho-1})^-_- \leq$$
$$\leq p(\varrho-1)-c-(p(\varrho-1)_{\varrho-1})^-_- + c=\Delta f^{\varrho-1}_{p(z)}(\varrho-1).$$
Now, the thesis is proved applying fact \emph{(i)}.

{\em (\ref{it:first derivative_iii})} Due to the definition of $\varrho_{\Delta p(z)}$, to Proposition \ref{prop:minimalfunction} and to fact {\em (\ref{it:first derivative_i})}, it is enough to show that the first derivative of $f^{\varrho}_{p(z)}$ 
is admissible in $\varrho$, i.e.~$(\Delta f^{\varrho}_{p(z)}(\varrho)_\varrho)^+_+ \geq \Delta f^{\varrho}_{p(z)}(\varrho+1)$. 
We know that $\Delta f^{r-1}_{p(z)}$ is admissible because $f^{r-1}_{p(z)}$ is the Hilbert function of the unique saturated lex-segment ideal
with Hilbert polynomial $p(z)$. 
So, the admissibility of $\Delta f^{\varrho}_{p(z)}$ in $\varrho$ follows from that of $\Delta f^{\varrho+1}_{p(z)}$. Indeed, by fact {\em (\ref{it:first derivative_ii})} we obtain
$(\Delta f^{\varrho}_{p(z)}(\varrho)_{\varrho})^+_+ \geq (\Delta f^{\varrho+1}_{p(z)}(\varrho)_{\varrho})^+_+ \geq 
\Delta f^{\varrho+1}_{p(z)}(\varrho+1) \geq \Delta p(\varrho+1)=\Delta f^{\varrho}_{p(z)}(\varrho+1)$.
\end{proof}

\begin{theorem}\label{th:first derivative}
\begin{enumerate}[(i)]
\item \label{it:th_first derivative_i} $\bar\varrho_{p(z)} = \max\{\varrho_{p(z)},\varrho_{\Delta p(z)}-1\}$ and $f^{\bar\varrho_{p(z)}}_{p(z)}\in F\big(p(z),\bar\varrho_{p(z)}\big)$;
\item \label{it:th_first derivative_ii} for all $\varrho>\bar\varrho_{p(z)}$, $$F(p(z),\varrho)\not= \emptyset \Leftrightarrow
\left\{\begin{array}{ll} f_{p(z)}^\varrho \in F(p(z),\varrho), &\text{ if $f^\varrho_{p(z)}$ has regularity } \varrho, \\
g_{p(z)}^\varrho \in F(p(z),\varrho), &\text{ otherwise.} 
\end{array}\right.$$
\end{enumerate}
\end{theorem}

\begin{proof}
\emph{(\ref{it:th_first derivative_i})} First, recall that $\bar\varrho_{p(z)}\geq \max\{\varrho_{p(z)},\varrho_{\Delta p(z)}-1\}$ by Proposition \ref{prop:varrho}. Then, by Lemma \ref{lemma:first derivative}{\em (\ref{it:first derivative_iii})} we obtain $\bar\varrho_{p(z)} = \max\{\varrho_{p(z)},\varrho_{\Delta p(z)}-1\}$ and $f^{\bar\varrho_{p(z)}}_{p(z)}$ belongs to $F(p(z),\bar\varrho_{p(z)})$, because the regularity of $f^{\bar\varrho_{p(z)}}_{p(z)}$ is $\bar\varrho_{p(z)}$ by the definition of $\bar\varrho_{p(z)}$.

\emph{(\ref{it:th_first derivative_ii})} The case involving $f^{\varrho}_{p(z)}$ follows from Lemma \ref{lemma:first derivative}. Thus, suppose that $f^{\varrho}_{p(z)}$ has regularity $\varrho'<\varrho$. If $g^{\varrho}_{p(z)}$ belongs to $F(p(z),\varrho)$, obviously $F(p(z),\varrho)$ is non-empty. Vice versa, let $g:=g^\varrho_{p(z)}$ and observe that $(\Delta g(t)_t)^+_+ \geq \Delta g(t+1)$, for every $t\geq \varrho+1$, by definition of $\bar \varrho_{p(z)}$, and for every $1\leq t\leq\varrho-2$, by arguments analogous to those of the proof of Lemma \ref{lemma:first derivative}.

It remains to prove 
$(\Delta g(\varrho)_\varrho)^+_+ \geq \Delta g(\varrho+1)$ and $(\Delta g(\varrho-1)_{\varrho-1})^+_+ \geq \Delta g(\varrho)$. 

For the degree $\varrho$, if $g'$ belongs to $F(p(z),\varrho)$, then $g(\varrho-1)\leq g'(\varrho-1)$, hence $\Delta g(\varrho)\geq \Delta g'(\varrho)$ and $\Delta g'$ admissible implies $\Delta g$ admissible in the degree $\varrho$. 

For the degree $\varrho-1$, we obtain $(\Delta f^{\varrho}_{p(z)}(\varrho-1)_{\varrho-1})^+_+\geq \Delta f^{\varrho}_{p(z)}(\varrho)=\Delta  g(\varrho)+1$ by Lemma \ref{lemma:first derivative} and by the definition of $g$. Moreover, by \eqref{disuguaglianze robbiano} 
$$\Delta g(\varrho-1) =  p(\varrho-1)+1-((p(\varrho-1)+1)_{\varrho-1})^-_- \geq $$
$$ \geq p(\varrho-1)+1-(p(\varrho-1)_{\varrho-1})^-_- -1 = p(\varrho-1)-(p(\varrho-1)_{\varrho-1})^-_-  = \Delta f^{\varrho}_{p(z)}(\varrho-1)$$
and $\Delta f^\varrho_{p(z)}$ admissible implies the first derivative of $g$ admissible in the degree $\varrho-1$. 
\end{proof}

\begin{example}\rm \label{ex:minime2}
Consider the admissible polynomial $p(z)=5z-3$ of Remark \ref{varrho'}(3) with Gotzmann number $r=7$ and $\varrho_{p(z)}=3$. We have that  $f_{p(z)}^3=f_{p(z)}^{4}=(1,4,8,12,p(z))$, meanwhile $f_{p(z)}^{4}\not= f_{p(z)}^{5}$. So, we can take the function $g_{p(z)}^{4}=(1,4,8,13,p(z))$, that is the minimal admissible function with regularity $4$ and Hilbert polynomial $p(z)$. But, its first derivative  $\Delta g_{p(z)}^{4}=(1,3,4,5,4,5)$ is not admissible. Thus, by Theorem \ref{th:first derivative} we can affirme that $F(p(z),4)=\emptyset$, i.e.~there are not schemes with Hilbert function having regularity $4$ and Hilbert polynomial $p(z)$. 
\end{example}

If the Gotzmann number $r$ of the polynomial $p(z)$ is big, a computation of $\varrho_{p(z)}$ and of $\bar\varrho_{p(z)}$ performed according to formula \eqref{varrho p(z)} and to Theorem \ref{th:first derivative}\emph{(\ref{it:th_first derivative_i})} can be very expensive. Indeed, we have to test ($p(t)_{t})^+_+ \geq p(t+1)$, for every $\varrho_{p(z)}-1\leq t\leq r-2$. For example, for the polynomial $p(z)=2z^3-6z^2+29z-20$ (see Example \ref{example m=7}) we find $r=218498$ and $\varrho_{p(z)}=\bar\varrho_{p(z)}=2$. Anyway, Proposition \ref{prop:varrho}\emph{(\ref{it:varrho_ii})} will help to compute $\varrho_{p(z)}$ and $\bar\varrho_{p(z)}$ in an efficient way, by induction on the degree of $p(z)$, as the following proposition shows using the features of minimal functions.

\begin{proposition}\label{prop:calcolo bar varrho}
If $\deg p(z) > 0$, $\bar\varrho_{p(z)}=\sigma-1$, with 
$\sigma:=\min\big\{t\geq \max\{\varrho_{\Delta p(z)},1\} : \sum f^t_{\Delta p(z)}(t-1) \leq  p(t-1) \big\}$. 
\end{proposition}

\begin{proof}
If $\varrho_{\Delta p(z)}=0$, then we can observe that $f^0_{\Delta p(z)}=f^1_{\Delta p(z)}$. Thus, in any case the Hilbert function $\sum f^\sigma_{\Delta p(z)}$ belongs to $F(p(z)-c,\sigma-1)$ for some non-negative integer $c$, by construction and according to Remark \ref{derivata-integrale}. Then, by Proposition \ref{prop:varrho}\emph{(\ref{it:varrho_ii})}, we have $\varrho_{p(z)}\leq \bar\varrho_{p(z)} \leq \bar\varrho_{p(z)-c} \leq \sigma-1$.

Now, observe that $\Delta f^{\bar\varrho_{p(z)}}_{p(z)}$ is an admissible function of regularity $\bar\varrho_{p(z)}+1$ and Hilbert polynomial $\Delta p(z)$. Thus, we obtain $f^{\bar\varrho_{p(z)}+1}_{\Delta p(z)}\leq \Delta f^{\bar\varrho_{p(z)}}_{p(z)}$, hence $\sum f^{\bar\varrho_{p(z)}+1}_{\Delta p(z)}\leq f^{\bar\varrho_{p(z)}}_{p(z)}$ and $\bar\varrho_{p(z)}+1\geq \sigma$.
\end{proof}

\begin{remark}\label{rk:calcolo varrho}
Proposition \ref{prop:calcolo bar varrho} gives a closed formula for $\bar\varrho_{p(z)}$ and also the opportunity to compute efficiently $\varrho_{p(z)}$. Indeed, as $\bar\varrho_{p(z)} = \max\{\varrho_{p(z)},\varrho_{\Delta p(z)}-1\}$, if $\sigma > \varrho_{\Delta p(z)}$ then $\varrho_{p(z)} = \bar\varrho_{p(z)} = \sigma-1$. Otherwise we have to check that $p(z)$ is admissible in the degrees $t\leq \sigma-2$, instead of $t\leq r-2$ as in the definition of $\varrho_{p(z)}$ of Proposition \ref{prop:varrho} (see Algorithms \ref{alg:barvarrho} and \ref{alg:varrho}).
\end{remark}


\section{Borel ideals and growth-height-lexicographic Borel sets}

The {\em Borel order} $<_B$ is the partial order on $\mathbb T$ for which, given two terms $x^\alpha$ and $x^\beta$, we have $x^\alpha <_B x^\beta$
 if there is a finite sequence of terms $\tau_1=x^\alpha,\tau_2,\ldots, \tau_k=x^\beta$ such that $\tau_h=\frac{x_i}{x_j}\tau_{h-1}$ with $j<i$,
 for every $1<h\leq k$. 

A set $B\subset \mathbb T_t$ is called a {\em Borel set} if, for every term $x^{\alpha}$ in $B$ and $x^{\beta}$ in $\mathbb T_t$, $x^{\alpha} <_B x^{\beta}$ implies that $x^{\beta}$ belongs to $B$.
From the definition it follows immediately that, if $B\subset \mathbb T_t$ is a Borel set, then the set $N:=\mathbb T_t\setminus B$ has the property that, for every $x^{\gamma}\in N$ and $x^{\delta}\in\mathbb T_t$ with $x^{\delta} <_B x^{\gamma}$,  $x^{\delta}$ belongs to $N$. 

A monomial ideal $J$ is {\em strongly stable} if $J_t$ is a Borel set, for each $t$. Every lex-segment ideal is a strongly stable ideal. 

A strongly stable ideal is always Borel-fixed (Borel, for short), i.e.~it is fixed under the action of the Borel subgroup of the upper-triangular invertible matrices. If $\ch(K) = 0$, also the vice versa holds. In $\ch(K)=0$ Galligo and in any characteristic Bayer and Stillman guarantee that in generic coordinates the initial ideal of an ideal $I$, with respect to (w.r.t.)~a given term order, is a constant Borel ideal, called the {\em generic initial ideal} of $I$. We denote by $\gin(I)$ the generic initial ideal of $I$ w.r.t.~the degrevlex order and recall that $\reg(I)=\reg(\gin(I))$ \cite{BS}.

\begin{proposition} \label{borel}
Let $J\subset S$ be a Borel ideal.
\begin{enumerate}[(i)]
\item The saturation $J^{\sat}$ of $J$ is the ideal generated by the terms minimal generators of $J$ in which we set the least variable equal to $1$ \cite[Proposition 2.9]{Gr}.
\item $reg(J)$ is the maximal degree of its minimal generators \cite[Proposition 2.9]{BS}.
\end{enumerate}
\end{proposition}

Now, we recall the notion of growth-height-lexicographic Borel set which has been introduced by D. Mall \cite{Mall1,Mall2}. 

\begin{definition}[{\cite[Definition 2.7]{Mall1}, \cite[Definition 2.7]{Mall2}}]
Let $B\subset \mathbb T_t$ be a Borel set. The set $B^{(i)}=\{\tau\in B : \min(\tau)=i\}$ is the {\em growth class} of $B$ of {\em growth} $i$.
The {\em growth-vector} of $B$ is $gv(B):=(\vert B^{(0)}\vert,\vert B^{(1)}\vert\ldots,\vert B^{(n)}\vert)$.
\end{definition}

\begin{definition}[{\cite[Def. 2.13]{Mall1}, \cite[Def. 2.8]{Mall2}}]
Let $B\subset\mathbb T_t$ be a Borel set. For every $i\in \{0,\ldots,t\}$, let $B(i):=\{x_0^{\alpha_0} x_1^{\alpha_1}\ldots x_n^{\alpha_n}\in B: \alpha_0=i\}$.
The {\em height-vector} of $B$ is the vector $hv(B):=(\vert B(0)\vert,\vert B(1)\vert,\ldots,$ $\vert B(t)\vert)$.
\end{definition}

Given a Borel set $B$, the {\em first expansion} of $B$ is $\{x_0,\ldots,x_n\}\cdot B$ 
(see \cite{Ma,MR,MR2} for a description of this topic). For a strongly 
stable ideal $J$, we call $\mathbb T_{j+1}\setminus (\{x_0,\ldots,x_n\}\cdot J_j)$ the {\em first expansion} of $\mathcal N(J)_j$.

\begin{remark}
Given a Borel set $B\subset\mathbb T_t$, we have:  for every $i\neq j$, $B^{(i)}\cap B^{(j)}=\emptyset$  and  $B{(i)}\cap B{(j)}=\emptyset$; $B(0) =\bigcup_{i\geq 1}B^{(i)}$; $B^{(0)} =\bigcup_{i\geq 1}B{(i)}$.
\end{remark}

\begin{proposition}[{\cite[Proposition 3.2]{Mall1}}] \label{height-vector}
Given a Borel set $B\subset\mathbb T_t$, let $I=(B)^{\sat}\subset S$ be the saturation of the homogeneous ideal generated by $B$. 
If $f$ is the Hilbert function of $S/I$, then:
\begin{enumerate}[(i)]
\item\label{it:height-vector_i} $f(j)=\binom{j+n}{n}-\sum_{i=t-j}^t \vert B(i)\vert$, for every $j\leq t$;
\item\label{it:height-vector_ii} $f(j)-f(j-1)=\binom{j+n-1}{n-1}-\vert B(t-j)\vert $  for every $j\leq t$;
\item\label{it:height-vector_iii} $f(t+k)=\binom{t+k+n}{n}-\sum \binom{i+k}{k} \vert B^{(i)}\vert$, for every $k\geq 1$.
\end{enumerate} 
\end{proposition}

\begin{definition}[{\cite[Def. 2.10 and 2.14]{Mall1}, \cite[Def. 4.3]{Mall2}}]\label{DefMall} 
\phantom f \ 

(1) \label{it:DefMall_i} The Borel set $B\subset\mathbb T_t$ is {\em growth-lexicographic} if all growth classes of $B$ are lex-segments in the corresponding growth classes $\mathbb T_t^{(i)}$ of $\mathbb T_t$.

(2) \label{it:DefMall_ii} The Borel set $B\subset\mathbb T_t$ is {\em growth-height-lexicographic} if $B(0)$ is growth-lexicographic and $B(i)$ is a lex-segment in $\mathbb T_t(i)$, for all $0<i\leq t$. 
\end{definition}

\begin{theorem}\label{th:crucial} 
Let $B\subset\mathbb T_t$ be a Borel set. Then there is a unique growth-height-lexicographic Borel set 
$L_{gh}(B)\subset \mathbb T_t$ with $gv(L_{gh}(B))=gv(B)$ and with $hv(L_{gh}(B))=hv(B)$.
Moreover, the ideals $J:=(B)^{\sat}$ and $I:=(L_{gh}(B))^{\sat}$  have the same Hilbert function and $\reg(J)\leq \reg(I) \leq t$.
\end{theorem}

\begin{proof} 
For the first part of this result we refer to \cite[Th. 2.17]{Mall1} and \cite[Theorem 4.4]{Mall2} to obtain that $L_{gh}(B)$ is the Borel set $L$ that consists of the union of the lex-segments $L^{(i)}$ in $\mathbb T_t^{(i)}$ with $\vert B^{(i)}\vert$ elements, for every $i=0,\dots,n$, and of the lex-segments $L(j)$ in $\mathbb T_t{(j)}$ with $\vert B(j)\vert$ elements, for every $j=1\dots,t$.  

For the second part, first we note that $H_{S/J}=H_{S/I}$ follows straightforwardly from Proposition \ref{height-vector}. Then, we observe that $\reg(I)\leq t$, because $I$ is generated by terms of degree $\leq t$ and is a strongly stable ideal. It remains to show that $\reg(J)\leq \reg(I)$.

Assume there are an integer $m>\reg(I)$ and a minimal generator of $J$ of degree $m$. Observe that $m \leq t$. Thus, every term of $I\cap K[x_1,\ldots,x_n]_m$ is multiple of a term of the set $L'$ of terms in $I\cap K[x_1,\ldots,x_n]_{m-1}$, while there are terms of $J\cap K[x_1,\ldots,x_n]_m$ that are not multiple of a term in the set $B'$ of the terms in $J\cap K[x_1,\ldots,x_n]_{m-1}$. Observe that the set of terms in $I\cap K[x_1,\ldots,x_n]_{h}$ is $L(t-h)\cdot x_0^{-t+h}$ and the set of terms in $J\cap K[x_1,\ldots,x_n]_{h}$ is $B(t-h)\cdot x_0^{-t+h}$, for every $h\leq t$. Moreover, $\vert L'\vert= \vert  L(t-m+1) \vert=\vert B(t-m+1) \vert=\vert B'\vert$ and $L'$ is a lex-segment. Then, we should have $\dim_K (B'\cdot R)_m< \dim_K (L'\cdot K[x_1,\ldots,x_n])_m$, but this is a contradiction with the fact that the first expansion of an ideal generated by a lex-segment is the smallest possible.  
\end{proof}

\section{Ideal graft and minimal Castelnuovo-Mumford regularities}

In this section, first we exploit the features of the growth-height-lexicographic Borel sets to construct a scheme $X$ obtained by a so-called \emph{\innesto} of two given schemes $X_1$ and $X_2$. Indeed, the Hilbert function of $X$ will be the ``graft" of the other two, because it will coincide with the Hilbert function of $X_1$ up to a certain degree and with the Hilbert function of $X_2$ from this degree on. 
Then, we apply this new construction to find a scheme $X$ such that the regularity $\varrho_X$ and the value $H_X(\varrho_X-1)$ are small, according to Lemma \ref{CMR} and Proposition \ref{newreg}, respectively. In conclusion, we prove that each $m^\varrho_{p(z)}$ is achieved by a scheme with the minimal function having regularity $\varrho$.

\begin{remark}\label{rem:immersione}
If $I\subset K[x_0, \dots, x_{n'}]$ is a strongly stable ideal, then for every integer $n>n'$ the ideal $(I + (x_{n'+1}, \dots, x_{n})) \subseteq K[x_0,\ldots,x_n]$ is a strongly stable ideal.
\end{remark}

\begin{theorem}[Ideal graft]\label{innesto}
Let $q,w$ be Hilbert functions of projective schemes and suppose $m>1$ is an integer such that $w(m-1)=q(m-1)$ and $w(m-2)\leq q(m-2)$ (so that $\Delta w(m-1)\geq \Delta q(m-1)$). Then, the function 
\[ 
h(j):=
\left\{\begin{array}{cc} w(j) & \hbox{ for all } j<m \\
q(j) & \hbox{ for all } j\geq m 
\end{array}\right. 
\] 
is the Hilbert function of a projective scheme and $m_h\leq \max\{m,m_q\}$. 
\end{theorem}

\begin{proof} 
Let $s:=\max\{m, m_q \}$. By the hypotheses, there are a saturated polynomial ideal $I$, defining a projective scheme with Hilbert function $q$ and Castelnuovo-Mumford regularity $m_q$, and a saturated polynomial ideal $I'$, defining a projective scheme with Hilbert function $w$. Replacing $I$ and $I'$ by $\gin(I)$ and $\gin(I')$, respectively, we can suppose that they are Borel saturated. Moreover, by Remark \ref{rem:immersione}, we can suppose that the ideals $I$ and $I'$ are in the same polynomial ring $S=K[x_0,\ldots,x_n]$. 
Nevertheless, we can suppose $\reg(I')\leq s$, replacing $I'$ possibly by $(I'_{\leq s})$, which gives a Hilbert function equal to $w$ at least until the degree $s$. Hence, we let $L:=L_{gh}(I_s)$ and $L':=L_{gh}(I'_s)$ and we can replace again $I$ and $I'$ by $(L)^{\sat}$ and $(L')^{\sat}$, respectively.

For each saturated Borel ideal $V\subset S$ with $H_{S/V}=v$ and integer $k$, we get
$\Delta v(k)=v(k)-v(k-1)=\binom{k+n-1}{n-1}-\vert V\cap \mathbb T_k(0) \vert.$
Then in our hypotheses, we have $\vert I'\cap \mathbb T_{m-1}(0)\vert \leq \vert  I\cap \mathbb T_{m-1}(0)\vert$, where by construction
$I\cap \mathbb T_{m-1}(0) =x_0^{-(s-m+1)} L(s-m+1)$ and $ I'\cap \mathbb T_{m-1}(0) =x_0^{-(s-m+1)} L'(s-m+1)$.
Recall that $L(s-m+1)$ e $L'(s-m+1)$ are both lex-segments in $\mathbb T_s(s-m+1)$, hence $I'\cap \mathbb T_{m-1}(0)\subseteq I\cap \mathbb T_{m-1}(0)$.

Let $G':=I'\cap \mathbb T_{\leq m-1}(0)$ and $G:=I\cap \mathbb T_{\geq m}(0)$ and then consider the ideal $J$ that is generated by the terms of the set $G'\cup G$.
First, we observe that:

i) $J$ is Borel and saturated, because it is generated by a suitable union of Borel sets of terms in which the variable $x_0$ does not occur; 

ii) $\reg(J)\leq s$, because $J$ is a Borel ideal generated by $G'\cup G_{\leq s}$, by the hypotheses on $I$ and $I'$.

It remains to show that $h=H_{S/J}$. By construction, for every $k\leq m-1$ we have $J_k=I_k$, hence $h(k)=w(k)$. 
By induction on $k$, we show that $h(k)=q(k)$, for every $k\geq m-1$. For $k=m-1$, we have $h(m-1)=w(m-1)=q(m-1)$, by the hypotheses. Now, let $k\geq m$ and assume the thesis is  true for $k-1$. 
We have $h(k) -h(k-1)=q(k)-q(k-1)$ because $J$ and $I$ contain the same terms of $\mathbb T_k(0)$. Then, by the inductive hypothesis we obtain $h(k)=q(k)$.
\end{proof}


\begin{corollary} \label{cor:uso innesto} Let $\varrho$ be any integer such that $F(p(z),\varrho)\neq \emptyset$ and let $w \in F(p(z),\varrho)$. 
For every $q\in  F(p(z))$, we get $m_w\leq \max\{\varrho+2,m_q+1\}$.
If moreover $q\in F(p(z),\varrho)$, then:
\begin{enumerate}[(i)]
\item\label{it:uso innesto_i}   $ m_q<m_w \Longleftrightarrow  m_q=\varrho+1,   \ m_w =\varrho+2 \hbox { and } q(\varrho-1)<w(\varrho-1)$;
\item\label{it:uso innesto_ii} either  $\vert M(p(z),\varrho)\vert =1$ or $M(p(z), \varrho)= \{\varrho+1, \varrho+2\}$;
\item\label{it:uso innesto_iv}  $m^{\varrho}_{p(z)}\geq m_{f^\varrho_{p(z)}}$ and  $m^{\varrho}_{p(z)}= \begin{cases} m_{f^\varrho_{p(z)}}, &\text{ if } f^\varrho_{p(z)} \text{ has regularity } \varrho \\
m_{g^\varrho_{p(z)}}, &\text{ otherwise.} 
\end{cases}$ 
\end{enumerate}
\end{corollary}

\begin{proof}
For the first assertion, it is enough to apply Theorem \ref{innesto} to the functions $w$ and $q$ with $m=\max\{\varrho_w+2,m_q+1\}$, obtaining an ideal $J$ such that $H_{S/J}=w$ and so $m_w\leq \reg(J)\leq \max\{m,m_q\}=\max\{\varrho_w+2,m_q+1\}$. 

\emph{(\ref{it:uso innesto_i})} Assume $m_q < m_w$ and set $m:=\varrho+2$. Then, $q(k)=w(k)=p(k)$ for $k=m-2$ and $k= m-1$, and we can apply Theorem \ref{innesto} to construct a saturated ideal $J\subset S$ with $H_{S/J}=w$ and $\reg(J)\leq \max\{m,m_q\}$. Hence, $m_w\leq \reg(J)\leq \max\{\varrho+2,m_q\}$ and we obtain $m_q=\varrho+1$, $m_w =\varrho+2$. Moreover, if on the contrary $w(\varrho-1)\leq q(\varrho-1)$, then we can apply Theorem \ref{innesto} with $m=m_q=\varrho+1$ obtaining an ideal $J\subset S$ such that $H_{S/J}=w$ and $m_w\leq \max\{m,m_q\}= m_q$, against the assumption.

Statement \emph{(\ref{it:uso innesto_ii})} is a straightforward consequence of part \emph{(\ref{it:uso innesto_i})}. 

For \emph{(\ref{it:uso innesto_iv})}, by the previous results $m^{\varrho}_{p(z)}$ is attained by the minimal function in $F(p(z), \varrho)$, hence by either $f_{p(z)}^\varrho$ or $g_{p(z)}^\varrho$. Then, we apply Theorem \ref{th:first derivative}.
\end{proof}

The results of Corollary \ref{cor:uso innesto} confirm that the Hilbert function of a scheme $X$ with minimal Castelnuovo-Mumford regularity must either have minimal regularity or assume the smallest possible value at the degree $\varrho_X-1$. 

\begin{remark}\label{M}
If $\vert M(p(z),\varrho)\vert=1$, then the unique element $m$ of $ M(p(z),\varrho)$ can be bigger than $\varrho+2$ (see Example \ref{example m=7}).
\end{remark}



\begin{example}
For the admissible polynomial $p(z)=12z-25$ we have $\bar\varrho_{p(z)}=\varrho_{p(z)}=6$ and $f^7_{p(z)}=f^6_{p(z)}=(1,4,9,16,25,36,12z-25)$. So, we consider $g^7_{p(z)}=(1,4,9,16,25,36,48,$ $12z-25)$ whose first derivative $\Delta g^7_{p(z)}=(1,3,5,7,9,11,12,11,12)$ is admissible. Thus, $F(p(z),7)\not=\emptyset$ and there are schemes with Hilbert function having regularity $7$ and Hilbert polynomial $p(z)$. Moreover, we obtain $m^7_{p(z)}=m_{g^7_{p(z)}}=9$ and $m^6_{p(z)}=m_{f^6_{p(z)}}=8$. 
\end{example}


\section{Expanded lifting} \label{sec:expandedLifting}

In this section, we use the notion of growth-height-lexicographic Borel set to obtain a scheme $X$, given the Hilbert function $f$ and a possible hyperplane section $Z$, where for a \lq\lq possible\rq\rq \ hyperplane section $Z$ we intend that $H_Z \leq \Delta f$ and $p_Z(z)=\Delta p(z)$.  
 
Our first tool is the following variant of \cite[Proposition 4.3]{CLMR}.

\begin{proposition}\label{DMvariant}
Let $J\subset S$ be a saturated strongly stable ideal with Hilbert polynomial $p(z)$ and $\reg(J)=m$. Let $x^\beta x_0^b$, with $x_0 \nmid x^\beta$ and $b> 0$, be a term of $J$ of degree $s\geq m$ which is minimal in $J_s$ w.r.t.~$<_B$. Then,
\begin{enumerate}[(i)]
\item\label{it:DMvariant_i} the ideal $I$ generated by $J_s\setminus \{x^\beta x_0^b\}$ is strongly stable and $p_{S/I}=p(z)+1$; in particular, its saturation $I^{\sat}$ is strongly 
stable with $p_{S/I^{\sat}}=p(z)+1$; 
\item\label{it:DMvariant_ii} $H_{S/I^{\sat}}(t)=H_{S/J}(t)$, for $t< \vert \beta \vert$, and $H_{S/I^{\sat}}(t)=H_{S/J}(t)+1$ otherwise; if $\vert \beta\vert=m$, then we get
$\reg(I^{\sat})=m+1$, otherwise $\reg(I^{\sat})=m$.
\end{enumerate}
\end{proposition}

\begin{proof} 
For part \emph{(\ref{it:DMvariant_i})}, we refer to the proof of \cite[Proposition 4.3]{CLMR}, in which we replace the Gotzmann number $r$ of $p(z)$ by the regularity $m$ of the ideal $J$. 
For part \emph{(\ref{it:DMvariant_ii})}, note that $x^\beta$ is a minimal generator of $J$, because $x^\beta x_0^h$ is a minimal term w.r.t.~$<_B$ in
$J_{\vert \beta \vert+h}$, for every $h\geq 0$. Then, $(I^{\sat})_{\vert\beta\vert +h}$ is generated by $(J_{\vert\beta\vert +h}\cap \mathbb T)\setminus 
\{x^\beta x_0^h\}$ and also the result about the Castelnuovo-Mumford regularity follows.
\end{proof}

If $J$ has suitable minimal generators, then Proposition \ref{DMvariant} can be used to {\em change} the Hilbert function of $S/J$ into a given other Hilbert function, by constructing a new saturated strongly stable ideal. 

\begin{proposition}\label{prop:costruzione}
Let $L \subset \mathbb T_{m}$ and $L' \subset \mathbb T_{m'}$ be two growth-height-lexicograph-ic Borel sets, with $m\leq m'$, and let $I:=(L)^{\sat}, I':=(L')^{\sat}\subset S$. If there is an integer $\bar t<m$ such that $\Delta H_{S/I}(t)=\Delta H_{S/I'}(t)$, for every $t<\bar t$, and $\Delta H_{S/I}(\bar t)<\Delta H_{S/I'}(\bar t)$, then $I$ has a minimal generator of degree $\bar t$. Hence, $L$ contains a term of type $x^\beta x_0^b$, with $\vert\beta\vert=\bar t$ and $x_0 \nmid x^\beta$ and $b> 0$,  which is minimal in $L$ w.r.t.~$<_B$.
\end{proposition}

\begin{proof} 
By the hypotheses and the definition of growth-height-lexicographic Borel set, we have $x_0^{-m+\bar t +1}(I_{\bar t-1}\cap \mathbb T)= x_0^{-m'+\bar t +1}(I'_{\bar t-1}\cap \mathbb T)$. So, the first expansions of $x_0^{-m+\bar t +1}(I_{\bar t-1}\cap \mathbb T)$ and of $x_0^{-m'+\bar t +1}(I'_{\bar t-1}\cap \mathbb T)$ are the same and the difference between the Hilbert functions implies the existence of a minimal generator of degree $\bar t$ for $I$, by Proposition \ref{height-vector}. The last assertion follows because now we can take the minimal term w.r.t.~$<_B$ among the terms of $L$ of type $x^\beta x_0^b$, with $\vert\beta\vert=\bar t$ and $x_0 \nmid x^\beta$ and $b> 0$.
\end{proof}

\begin{example}\label{ex:construction3}
Consider the polynomial $p(z)=9z-7$. Observe that for the saturated strongly stable ideal $J := (x_4^2,x_4x_3,x_3^2,x_4x_2,x_3x_2^3,x_2^5)\subset S=K[x_0,\ldots,x_4]$
we get $H_{S/J}=(1,5,11,9z-8)\preceq f^2_{p(z)}=(1,5,9z-7)$, but from $J$ we cannot construct a saturated strongly stable ideal $I$ such that $S/I$ has Hilbert function $(1,5,9z-7)$ by Proposition \ref{DMvariant} because $J$ has not a minimal generator of degree $3$.
Anyway, taking the Borel set $B:=J_5\cap \mathbb T$ with growth vector $gv(B)=(42,26,15,5,1)$ and height vector $hv(B)=(47,26,12,4)$, we can consider the growth-height-lexicographic Borel set $L:=L_{gh}(B)$ that generates an ideal whose saturation $I:=(x_4^2,x_3x_4,x_2x_4,x_1x_4,x_3^3,x_2x_3^2,x_1^3x_3^2,$ $x_2^4x_3,x_2^5)$ 
has at least a generator of degree $3$, as we expect by Proposition \ref{prop:costruzione}.
\end{example}

\begin{theorem}[Expanded lifting]\label{construction} 
Let $f$ be a function of $F(p(z),\varrho)$ and $Z$ be a scheme with Hilbert polynomial $\Delta p(z)$ and a Hilbert function $g$ such that $g\preceq \Delta f$. Then there is a scheme $X$ such that $H_X=f$ and $\reg(X)=\max\{\reg(Z),\varrho+1\}$. 
\end{theorem}

\begin{proof} 
Consider $\gin(I(Z))\subset K[x_1,\ldots,x_n]$ and its lifting $J=\gin(I(Z))\cdot K[x_0,x_1,\ldots,x_n]$, that is a saturated strongly stable ideal with regularity $\reg(J)=\reg(Z)$ and $H_{S/J}=\sum g$. 

Take $m:=\max\{\reg(Z),\varrho+1\}$, $L:=L_{gh}(J_{m})$ and let $I:=(L)^{\sat}$. By Theorem \ref{th:crucial}, we have $\reg(Z)=\reg(J)\leq \reg(I)\leq m$. Moreover, by construction and by Proposition \ref{height-vector}, the Hilbert function of $S/I$ is the same as the Hilbert function of $S/J$, that is $\sum g\preceq f$, with Hilbert polynomial $p(z)-c$, for a non-negative integer $c$ (see Remark \ref{derivata-integrale}). 

If $\sum g= f$, then it is enough to let $X=\Proj S/I$, because in this case $\varrho+1 \leq \reg(I)\leq m$ and $\reg(Z)=\reg(J)\leq \reg(I)\leq m$, hence $\reg(I)=m$. 

Otherwise, let $Y$ be a scheme with $H_Y=f$ and let $J'= \gin(I(Y))$. Take $m'=\max\{m,\reg(Y)\}$, $L':=L_{gh}(J'_{m'})$ and let $I':=(L')^{\sat}$. As before, by construction and by Proposition \ref{height-vector}, the Hilbert function of $S/I'$ is the same as the Hilbert function of $S/J'$, that is $f$. By Theorem \ref{th:crucial}, we have $\reg(Y)=\reg(J')\leq \reg(I')\leq m'=\max\{\reg(Z),\varrho+1,\reg(Y)\}$, hence $\reg(Y)\leq \reg(I')\leq m$. 

Consider $\bar t:=\min\{t : \sum g(t) < f(t)\}$ and observe that 
$\bar t=\min\{t : g(t) < \Delta f(t)\} < m$, 
because $g(m)=\Delta p(m)=\Delta f(m)$ by the definition of $m$.

By Proposition \ref{prop:costruzione} applied to $L$ and $L'$, there is a minimal term $x^\beta x_0^{m-\beta}$ w.r.t.~$<_B$ belonging to $L=I_m$, with $\vert\beta\vert=\bar t$ and $x_0 \nmid x^\beta$ and $b> 0$. As in Proposition \ref{DMvariant}, we consider $I_m\setminus\{x^\beta x_0^{m-\beta}\}=L\setminus\{x^\beta x_0^{m-\beta}\}$, which generates an ideal whose saturation $\bar I$ has $H_{S/\bar I}(t)=H_{S/I}(t)=\sum g(t)=f(t)$, for $t< \bar t$, and $H_{S/\bar I}(t)=H_{S/I}(t)+1=\sum g(t)+1 \leq f(t)$ otherwise (see also Remark \ref{derivata-integrale}). In particular, $\Delta H_{S/\bar I}(t)=g(t)$, for $t\not= \bar t$, and $\Delta H_{S/\bar I}(\bar t)= g(\bar t)+1$. Moreover, $\reg(Z)\leq \reg(\bar I)\leq m$ because $\reg(\bar I)=\reg(I)$.

Thus, if $H_{S/\bar I}=f$, then it is enough to let $X=\Proj S/\bar I$, because in this case we also have $\varrho+1 \leq \reg(\bar I)\leq m$, hence $\reg(\bar I)=m$. 
Otherwise, we can repeat the above arguments on $\bar I$, redefining $g$ as $g:=\Delta H_{S/\bar I}$ and noticing that we have again $\Delta p(m)=g(m)=\Delta f(m)$.
\end{proof}

\begin{example}\label{ex:construction2}
For the polynomial $p(z)=15z-24$ we have $\bar\varrho_{p(z)}=\varrho_{p(z)}=3$. Consider the functions $f=(1,5,11,15z-24)\in F(p(z),3)$ and $g=(1,3,6,10,$ $15,\ldots)\preceq \Delta f=(1,4,6,10,15,15,\ldots)$. The saturated strongly stable ideal $J'=(x_3^5,x_3^4x_2, x_2^2x_3^3,x_2^3x_3^2,$ $x_3x_2^4,x_2^5,x_4)\subset K[x_1,\ldots,x_4]$ defines a scheme $Z\subset\mathbb P^3_K$  with Castelnuovo-Mumford regularity $\reg(Z)=5$. If $J:=J'\cdot K[x_0,\ldots,x_4]$ is a lifting of $J$, then the Hilbert function of $K[x_0,\ldots,x_4]/J$ is $\sum g=(1,4,10,20,15z-25)\preceq f$. In this case  we can apply Proposition \ref{DMvariant} with $\bar t=1$, taking the saturation $\bar J=(x_4^2,x_4x_3,x_4x_2,x_4x_1,x_3^5,x_3^4x_2,x_3^3x_2^2,$ $x_3^2x_2^3,x_3x_2^4,x_2^5)$ of the ideal generated by $(J)_5 \setminus \{x_0^4x_4\}$, that defines a curve in $\mathbb P^4_K$ with $\reg(\bar I)=5=\varrho_{p(z)}+2$ and $H_{S/\bar I}=f$. 
\end{example}


\section{Minimal Castelnuovo-Mumford regularity}

As announced in Notation \ref{varie}, here we set:
\begin{equation}\label{barra m}
\bar f:=f_{p(z)}^{\bar\varrho_{p(z)}} \text{ and } \bar m:=m_{\bar f}.
\end{equation}

In this section, we apply the constructive result of Theorem \ref{construction} to show that $m_{p(z)}=\bar m$ and to compute $m_u$, for every $\varrho\geq \bar\varrho_{p(z)}$ and $u\in F(p(z),\varrho)$. We also compute the minimal possible Castelnuovo-Mumford regularity of any scheme with Hilbert polynomial $p(z)$ embedded in a given projective space $\mathbb P^n_K$.
 
\begin{theorem}\label{primo teorema}
Let $u\in F(p(z),\varrho)$ and $v\in F(p(z),\varrho')$. Then,
$$\varrho < \varrho' \quad \Longrightarrow \quad m_u \leq m_v  \quad (\text{in particular  } \minRho{p(z)}{\varrho} \leqslant \minRho{p(z)}{\varrho'})$$
Hence $m_{p(z)}=\bar m$.
\end{theorem}

\begin{proof} First, we prove that $m_{f^{\varrho-1}_{p(z)}}\leq m_{f^{\varrho}_{p(z)}}$, for every $\varrho > \bar\varrho_{p(z)}$. Suppose $f^{\varrho}_{p(z)}\not= f^{\varrho-1}_{p(z)}$, hence $f^{\varrho}_{p(z)}$ has regularity $\varrho$. 
By Lemma \ref{lemma:first derivative}\emph{(ii)}, we have $\Delta f^{\varrho}_{p(z)}(t)\leq \Delta f^{\varrho-1}_{p(z)}(t)$ for every $t<\varrho$.
Anyway, we have also $\Delta f^{\varrho}_{p(z)}(\varrho)> \Delta f^{\varrho-1}_{p(z)}(\varrho)=\Delta p(\varrho)$, because 
$f^{\varrho}_{p(z)}(\varrho-1)<p(\varrho-1)$ since $\varrho>\varrho_{p(z)}$, by Proposition \ref{prop:minimalfunction}.

Let $X$ be a scheme with $H_X=f^{\varrho}_{p(z)}$ and $\reg(X)=m_{f^{\varrho}_{p(z)}}$, and let $g\in F(\Delta p(z),\tilde\varrho)$ be the Hilbert function of its general hyperplane section. Thus, we obtain $g\preceq \Delta f^{\varrho}_{p(z)}$ and $m_g\leq m_{f^{\varrho}_{p(z)}}$, 
by Lemma \ref{CMR}. There are two possible cases: either $g\preceq \Delta f^{\varrho-1}_{p(z)}$ or $g\not\preceq \Delta f^{\varrho-1}_{p(z)}$.
If $g\preceq \Delta f^{\varrho-1}_{p(z)}$, then by Theorem \ref{construction}, we have 
$m_{f^{\varrho-1}_{p(z)}}\leq \max\{m_g,\varrho\}\leq m_{f^{\varrho}_{p(z)}}$. 
If $g\not\preceq \Delta f^{\varrho-1}_{p(z)}$, by the described behavior of the first derivatives of our minimal functions, we obtain
$\Delta f^{\varrho}_{p(z)}(\varrho)\geq g(\varrho) > \Delta f^{\varrho-1}_{p(z)}(\varrho)=\Delta p(\varrho)$. Hence, the regularity of 
$g$ is $\geq \varrho+1$ and $m_{f^{\varrho}_{p(z)}}\geq m_g\geq \varrho+2$. This fact implies the thesis. Indeed, if on the contrary 
$m_{f^{\varrho-1}_{p(z)}}> m_{f^{\varrho}_{p(z)}}$, then we could apply Theorem \ref{innesto} to $w:=f^{\varrho-1}_{p(z)}$, $q=f^{\varrho}_{p(z)}$ and 
$m:=\varrho+2$ getting $m_{f^{\varrho-1}_{p(z)}} =\varrho+2$ and  $m_{f^{\varrho}_{p(z)}}=\varrho+1$, meanwhile $m_{f^{\varrho}_{p(z)}}\geq \varrho+2$. 

Now, if $m_u > m_v$ then, applying Theorem \ref{innesto} to $w:=u$, $q:=v$ and $m:=\varrho'+2$, we get $\varrho'+1 \leq m_v < m_u=\varrho'+2$,  hence $m_v=\varrho'+1$ and $m_u=\varrho'+2$. We have also $\varrho+1\leq m^\varrho_{p(z)}\leq m^{\varrho'}_{p(z)}\leq m_v=\varrho'+1$, hence $M(p(z),\varrho)=\{m^\varrho_{p(z)},m_u\}$ with $m^\varrho_{p(z)}\leq\varrho'+1$ and $m_u=\varrho'+2>\varrho+2$, against Corollary \ref{cor:uso innesto}\emph{(\ref{it:uso innesto_i})}.
\end{proof}

\begin{corollary}\label{cor:finale innesto}
Let $\varrho$ be any integer such that $F(p(z),\varrho)\neq \emptyset$. Then,  
\begin{equation} \label{formula finale innesto}
m^{\varrho}_{p(z)}= \begin{cases}  \bar m, &\text{ if } \bar\varrho_{p(z)} \leq \varrho \leq \bar m-2, \\
 m_{f^\varrho_{p(z)}}=\varrho+1,  & \text{ if } \bar m-1 \leq \varrho 
\text{ and } f^\varrho_{p(z)} \text{ has regularity }  \varrho, \\
m_{g^\varrho_{p(z)}}= \varrho+2,  &
  \text{ otherwise}. 
\end{cases}
\end{equation}
\end{corollary}

\begin{proof} 
Consider ${\varrho}>\bar\varrho_{p(z)}$, being the case ${\varrho}=\bar\varrho_{p(z)}$ obvious, and let $u$ be the minimal function in $F(p(z),\varrho)$, so $m^{\varrho}_{p(z)}=m_u$ by Corollary \ref{cor:uso innesto}\emph{(\ref{it:uso innesto_iv})}. By Theorem \ref{primo teorema}, we have $m_u\geq \bar m$. Recall that $\bar m>1$, because we are not considering linear varieties. 

By Corollary \ref{cor:uso innesto}\emph{(\ref{it:uso innesto_iv})}, there are two possibilities for $u$. If $u=f_{p(z)}^\varrho$, we can apply Theorem \ref{innesto} to $w:=u$, $q:=\bar f$ and $m:=\max\{\varrho+1,\bar m\}$ by the last part of Proposition \ref{prop:minimalfunction}, obtaining $h=w=u$ and $m_u\leq m$, hence $m_u=\varrho+1$.
If $u=g_{p(z)}^\varrho$, then $m_u \geq \varrho+2$ by Proposition \ref{newreg}, because $g_{p(z)}^\varrho(\varrho-1)>p(\varrho-1)$ by construction. We apply Theorem \ref{innesto} to $w:=u$, $q:=\bar f$ and $m:=\max\{\varrho+2,\bar m\}$, obtaining $h=w=u$ and $m_u \leq m$, hence $m_u=\varrho+2$.  
\end{proof} 

If $p(z)$ is an admissible polynomial for subschemes of $\mathbb P^n_K$, then we can pose the same questions about the Castelnuovo Mumford regularity for a given  dimension $n$ of the projective space, and we provide the following answer.

\begin{corollary}\label{domanda di Paolo}
The minimal possible Castelnuovo-Mumford regularity of a scheme $X\subset \mathbb P^n_K$, with $p_X(z)=p(z)$, is $m_{f_{p(z)}^{\varrho_n}}$ with $\varrho_n:=\min\{\varrho\geq \bar\varrho_{p(z)} : f_{p(z)}^{\varrho}(1)\leq n+1\}$.
\end{corollary}

\begin{proof}
Observe that $n\geq f^{r-1}_{p(z)}(1)-1$, because $f^{r-1}_{p(z)}$ is the minimum of the Hilbert functions with Hilbert polynomial $p(z)$. Then, apply Theorem \ref{primo teorema}.
\end{proof}



Finally, we have a recursive procedure to compute $m_u$, for every $u\in F(p(z),\varrho)$, and we get also lower and upper bounds in terms of $\bar\varrho_{\Delta^i p(z)}$.

\begin{theorem}\label{construction-main}
Let $k$ be the degree of the polynomial $p(z)$, $u\in F(p(z),\varrho)$ and 
\[
M:=\max_{1\leq i \leq k} \{\varrho, \bar \varrho_{\Delta^i p(z)} \},\ \tilde\varrho_u :=\min\{t\geq \bar\varrho_{\Delta p(z)} : f^t_{\Delta p(z)}\preceq \Delta u\},\ \tilde m:=m_{f^{\tilde\varrho_u}_{\Delta p(z)}}.
\]
If $k=0$, then $m_u=\varrho+1$. If $k>0$, then $m_{u} = \max\{\tilde m,\varrho+1\}$. Moreover, we obtain $M+1 \leq  m_{u} \leq M+2$.
\end{theorem}

\begin{proof}
The case $k=0$ holds by Remark \ref{varrho'}(2). For $k>0$, observe that $\tilde m$ is the lowest Castelnuovo-Mumford regularity of a general hyperplane section of a scheme with Hilbert function $u$, by Lemma \ref{CMR} and Theorem \ref{primo teorema}, hence $m_u\geq \tilde m$. Now, by Theorem \ref{construction}, we can construct a scheme $X$ with $H_X=u$ and $\reg(X)=\max\{\tilde m,\varrho+1\}=m_u$.

For the inequalities $M+1 \leq  m_{u} \leq M+2$ we argue by induction on $k$. The first inequality follows again from Lemma \ref{CMR} and from the definition of $\bar \varrho_{\Delta^i p(z)}$. For the second one, observe that if $m_u\geq\varrho+2$, then $m_u=\tilde m$, by the previous results. Moreover, by Corollary \ref{cor:uso innesto}\emph{(\ref{it:uso innesto_iv})}, $\tilde m=m^{\tilde \varrho_u}_{\Delta p(z)}$ because $f_{\Delta p(z)}^{\tilde\varrho_u}$ has regularity $\tilde\varrho_u$. By induction, we obtain $m_{\Delta p(z)} \leq \max_{1\leq i \leq k} \{\bar \varrho_{\Delta^i p(z)}\}+2$, because, by Theorem \ref{primo teorema}, $m_{\Delta p(z)}$ is the minimum regularity of a scheme with Hilbert function $f_{\Delta p(z)}^{\bar\varrho_{\Delta p(z)}}$, that has regularity $\bar\varrho_{\Delta p(z)}$. Moreover, by Theorem \ref{primo teorema} and Corollary \ref{cor:finale innesto}, we have $\tilde m \leq \max\{m_{\Delta p(z)},\tilde\varrho_u+1\}$. Thus, we complete the proof because $\tilde \varrho_u \leq \max\{\bar\varrho_{\Delta p(z)},\varrho+1\}$, by definition.
\end{proof}

\appendix\section{Algorithms and examples}

In this appendix, we collect the main algorithms that arise from the results we have described in our exposition. A trial version of these algorithms is available at the web page \href{http://www.personalweb.unito.it/paolo.lella/HSC/Minimal_Hilbert_Functions_and_CM_regularity.html}{\texttt{http://www.personalweb.unito.it/paolo.lella/HSC/Minimal\_Hilbert\_Functions\_\ }} \href{http://www.personalweb.unito.it/paolo.lella/HSC/Minimal_Hilbert_Functions_and_CM_regularity.html}{\texttt{and\_CM\_regularity.html}}.

\begin{algorithm}[H]
\caption{\label{alg:barvarrho} \textsc{RhoBar} computes the minimal regularity of a Hilbert function of a scheme with Hilbert polynomial $p(z)$ (see Proposition \ref{prop:calcolo bar varrho}).}
\begin{algorithmic}[1]
\STATE $\textsc{RhoBar}\big(p(z)\big)$
\REQUIRE $p(z)$ a Hilbert polynomial.
\ENSURE $\bar{\varrho}_{p(z)}$, the minimal regularity of a Hilbert function $f$ of a scheme with Hilbert polynomial $p(z)$. 
\IF{$\deg p(z) = 0$}
\RETURN $1$;
\ELSE 
\STATE $\varrho_{\Delta p(z)} \leftarrow \textsc{RhoMin}\big(\Delta p(z)\big)$;
\RETURN $\min\left\{t \geqslant \max\{\varrho_{\Delta p(z)},1\}\ \big\vert\ \Sigma f^{t}_{\Delta p(z)} (t-1) \leqslant p(t-1)\right\} - 1$;
\ENDIF
\end{algorithmic}
\end{algorithm}

\begin{algorithm}[H]
\caption{\label{alg:varrho} \textsc{RhoMin} computes the minimal regularity of a Hilbert function with Hilbert polynomial $p(z)$ (see Remark \ref{rk:calcolo varrho}).}
\begin{algorithmic}[1]
\STATE \textsc{RhoMin}($p(z)$)
\REQUIRE $p(z)$ a Hilbert polynomial.
\ENSURE $\varrho_{p(z)}$, the minimal regularity of a Hilbert function $f$ with $f(t) = p(t),\ t \gg 0$. 
\IF{$\deg p(z) = 0$}
\RETURN $1$;
\ELSE 
\STATE $\varrho_{\Delta p(z)} \leftarrow \textsc{RhoMin}\big(\Delta p(z)\big)$;
\STATE $\varrho \leftarrow \min\left\{t \geqslant \max\{\varrho_{\Delta p(z)},1\}\ \big\vert\ \Sigma f^{t}_{\Delta p(z)} (t-1) \leqslant p(t-1)\right\} - 1$;
\IF{$\varrho > \varrho_{\Delta p(z)}-1$}
\RETURN $\varrho$;
\ELSE
\WHILE{$\varrho>1$ and $(p(\varrho-1)_{\varrho-1})^{+}_{+} \geqslant p(\varrho)$}
\STATE $\varrho \leftarrow \varrho-1$;
\ENDWHILE
\IF {$\varrho=1$ and $p(0)=1$}
\STATE $\varrho \leftarrow 0$;
\ENDIF
\RETURN $\varrho$;
\ENDIF
\ENDIF
\end{algorithmic}
\end{algorithm}

\begin{algorithm}[H]
\caption{This function computes the minimal Castelnuovo-Mumford regularity of a scheme, given the Hilbert polynomial and the regularity of the Hilbert function (see Theorem \ref{construction-main}).}
\begin{algorithmic}[1]
\STATE \textsc{MinimalCMregularity}$\big(p(z),\varrho\big)$
\REQUIRE $p(z)$, a Hilbert polynomial.
\REQUIRE $\varrho$, an integer such that $F(p(z),\varrho)\not=\emptyset$.
\ENSURE the minimal Castelnuovo-Mumford regularity $m^{\varrho}_{p(z)}$ of a scheme with Hilbert function of regularity $\varrho$ and Hilbert polynomial $p(z)$. 
\IF{$\deg p(z) = 0$}
\RETURN $\varrho+1$;
\ELSE
\STATE $\bar{\varrho}_{\Delta p(z)} \leftarrow \textsc{RhoBar}\big(\Delta p(z)\big)$;
\STATE $b \leftarrow \min\{ t \geqslant \bar{\varrho}_{\Delta p(z)}\ \vert\ f^t_{\Delta p(z)} \preceq \Delta f^{\varrho}_{p(z)}\}$;
\RETURN $\max\left\{\varrho+1,\textsc{MinimalCMregularity}\big(\Delta p(z),b\big)\right\}$;
\ENDIF
\end{algorithmic}
\end{algorithm}

\begin{algorithm}[H]
\caption{This function computes the minimal Castelnuovo-Mumford regularity of a scheme, given the Hilbert polynomial (see Theorem \ref{construction-main}).}
\begin{algorithmic}[1]
\STATE \textsc{MinimalCMregularity}$\big(p(z)\big)$
\REQUIRE $p(z)$, a Hilbert polynomial.
\ENSURE the minimal Castelnuovo-Mumford regularity $m_{p(z)}$ of a scheme with Hilbert polynomial $p(z)$. 
\RETURN $\textsc{MinimalCMregularity}\Big(p(z),\textsc{RhoBar}\big(p(z)\big)\Big)$;
\end{algorithmic}
\end{algorithm}

\begin{algorithm}[H]
\caption{This function computes the minimal Castelnuovo-Mumford regularity $m_u$ of a scheme, given the Hilbert function $u$ with Hilbert polynomial $p(z)$ (see Theorem \ref{construction-main}).} \label{alg:u}
\begin{algorithmic}[1]
\STATE \textsc{MinimalCMregularity}$\big(p(z),u\big)$
\REQUIRE $p(z)$, a Hilbert polynomial.
\REQUIRE $u$, Hilbert function of a scheme with Hilbert polynomial $p(z)$.
\ENSURE the minimal Castelnuovo-Mumford regularity $m_u$ of a scheme with Hilbert function $u$. 
\STATE $\varrho \leftarrow$ regularity of $u$;
\STATE $b \leftarrow \min \left\{ t \geqslant \bar{\varrho}_{p(z)}\ \vert\ f^{t}_{\Delta p(z)} \preceq \Delta u \right\}$
\RETURN $\max\left\{\varrho+1,\textsc{MinimalCMregularity}\big(\Delta p(z),b\big)\right\}$;
\end{algorithmic}
\end{algorithm}

\begin{algorithm}[H]
\caption{This function computes the minimal Castelnuovo-Mumford regularity of a subscheme of $\PP^n$ with Hilbert polynomial $p(z)$ (see Corollary \ref{domanda di Paolo}).}
\begin{algorithmic}[1]
\STATE \textsc{MinimalCMregularity}$\big(p(z),\PP^n\big)$
\REQUIRE $p(z)$, a Hilbert polynomial.
\REQUIRE $\PP^n$, a projective space with $n \geqslant \deg p(z) + 1$.
\ENSURE the minimal Castelnuovo-Mumford regularity of a subscheme of $\PP^n$ with Hilbert polynomial $p(z)$. 
\STATE $\bar{\varrho}_{p(z)} \leftarrow \textsc{RhoBar}\big(p(z)\big)$;
\STATE $\varrho \leftarrow \min \left\{ t \geqslant \bar{\varrho}_{p(z)}\ \vert\ f^{t}_{p(z)}(1) \leqslant n+1\right\}$
\RETURN $\textsc{MinimalCMregularity}\big(p(z),\varrho\big)$;
\end{algorithmic}
\end{algorithm}

\begin{example}\label{esempio Paolo}
We compute $m_{p(z)}$ for the Hilbert polynomial $p(z)=\frac{1}{3}z^3+2z^2+\frac{14}{3}z-4$, applying \textsc{MinimalCMregularity}$\big(p(z)\big)$. We get $\Delta p(z)=z^2+3z+3$, $\Delta^2 p(z)=2z+2$ and $\Delta^3 p(z)=2$ (see Table \ref{tab:1}). Moreover, we have $\varrho_{p(z)}=3$ and $\varrho_{\Delta p(z)}=1$, hence $\bar\varrho_{p(z)}=\varrho_{p(z)}=3$. For $u:=f^3_{p(z)}=(1,6,17,p(z))$, we compute $\Delta f^3_{p(z)}=(1,5,11,20,\Delta p(z))$ and have $\tilde\varrho_u=4$. For $u:=f^4_{\Delta p(z)}=(1,4,10,19,\Delta p(z))$, we compute $\Delta f^4_{\Delta p(z)}=(1,3,6,9,12,12,$ $\Delta^2 p(z))$ and get $\tilde\varrho_u=2$, with $f^2_{\Delta^2 p(z)}=(1,3,\Delta^2 p(z))$. Since $m_{f^2_{\Delta^2 p(z)}}=3$, we get $m_{f^4_{\Delta p(z)}}=5$, hence $m_{p(z)}=5$.
\begin{table}
\caption{Example \ref{esempio Paolo}}
\label{tab:1}
\begin{tabular}{l|c|c|c|c|c|c}
\hline 
\noalign{\smallskip}
$p(z)$ & $r$ & $\varrho_{p(z)}$ & $\bar\varrho_{p(z)}$ &$\varrho$ & $\tilde\varrho$ & $m^{\varrho}_{p(z)}$\\
\noalign{\smallskip}\hline\noalign{\smallskip}
 $\frac{1}{3}z^3+2z^2+\frac{14}{3}z-4$ & $10$ & $3$ & $3$ & $3$ & $4$ & $5$\\
 $z^2+3z+3$  & $6$ & $1$ & $1$ & $4$ & $2$ & $5$ \\
 $2z+2$  & $3$ & $1$ & $1$ & $2$ & $1$ & $3$\\
 $2$  & $2$ & $1$ & $1$ & $1$ & &  $2$\\
\noalign{\smallskip}
\hline
\end{tabular}
\end{table}
\end{example}

\begin{example} \label{example m=7}
For the admissible polynomial $p(z)=6z^2-18z+37$ we already considered in Example \ref{ex:first derivative}, we get $\Delta p(z)=12z-24$, $\varrho_{p(z)}=1$ and $\varrho_{\Delta p(z)}=5$, hence $\bar\varrho_{p(z)}=4$ (see Table \ref{tab:2}). For $u:=f^4_{p(z)}=(1,5,15,33,p(z))$, we compute   $\Delta f^4_{p(z)}=(1,4,10,18,28,\Delta p(z))$ and have $\tilde\varrho_u=5$. Then, for  $u:=f^5_{\Delta p(z)}=(1,4,9,16,25,\Delta p(z))$ we compute $\Delta f^5_{\Delta p(z)}=(1,3,5,7,9,11,$ $\Delta^2 p(z))$ and get $\tilde\varrho_u=6$, hence $m_{p(z)}=7$.

For the admissible polynomial $q(z)=2z^3-6z^2+29z-20$, we have $\Delta q(z)=p(z)$ and $\varrho_{q(z)}=2$. So, we take $f^2_{q(z)}=(1,8,q(z))$ and $\Delta f^2_{q(z)}=(1,7,22,\Delta q(z))$. By $\bar\varrho_{\Delta q(z)}=4$, we get $\tilde\varrho_u=4$ for $u:=f^4_{p(z)}=f^4_{\Delta q(z)}\preceq f^3_{\Delta q(z)}\preceq \Delta f^2_{q(z)}$. Using the already computed $m_{f^{4}_{p(z)}}=m_{p(z)}$, we have also $m_{q(z)}=7$.
\begin{table}
\caption{Example \ref{example m=7}}
\label{tab:2}
\begin{tabular}{l|c|c|c|c|c|c}
\hline 
\noalign{\smallskip}
$p(z)$ & $r$ & $\varrho_{p(z)}$ & $\bar\varrho_{p(z)}$ &$\varrho$ & $\tilde\varrho$ & $m^{\varrho}_{p(z)}$\\
\noalign{\smallskip}\hline\noalign{\smallskip}
 $2z^3-6z^2+29z-20$ & $218498$ & $2$ & $2$ &$2$ & $4$ & $7$\\
 $6z^2-18z+37$ & $678$ & $1$ & $4$ & $4$ & $5$  & $7$ \\
 $12z-24$  & $42$ & $5$ & $5$ & $5$ & $6$ & $7$ \\
 $12$  & $12$ & $1$ & $1$ & $6$ & & $7$ \\
\noalign{\smallskip}
\hline
\end{tabular}
\end{table}
\end{example}



\end{document}